\documentclass{amsart}
\usepackage{graphicx} 
\usepackage{amsmath, amsthm, amssymb}
\usepackage{tikz-cd}
\usepackage{tikz}
\usetikzlibrary{positioning}
\usepackage{stmaryrd}
\usepackage{relsize}
\usepackage{enumitem} 
\usepackage{ragged2e}
\usepackage{microtype}
\usepackage{amsfonts}
\usepackage{extarrows}

\usepackage{placeins}
\usepackage{array}
\usepackage{listings}
\usepackage{makecell}
\usepackage{lmodern}
\usepackage{mathtools}
\usepackage{setspace}
\usepackage{longtable}  
\usepackage{multirow}
\usepackage{caption}  
\usepackage{comment}

\usepackage{csquotes}
\usepackage{longtable}
\usepackage{array}
\usepackage{capt-of}
\usepackage{MnSymbol}
\usepackage{lscape}
\usepackage{float}
\usepackage{hyperref}
\usepackage{url}
\usepackage{rotating}
\usepackage{color}
\usepackage{youngtab}
\usepackage{ytableau}
\usepackage{tabularx}

\DeclareMathOperator{\Hom}{Hom}

\DeclareMathOperator{\Ext}{Ext}
\DeclareMathOperator{\GL}{GL}
\DeclareMathOperator{\Gr}{Gr}
\DeclareMathOperator{\End}{End}

\title[Koszul Artin Schelter-Regular Algebras of Dimension four]{Discrete Invariants of Koszul Artin-Schelter Regular Algebras of Dimension four}

\author{Vishal Bhatoy}
\address{Department of Mathematics, Carleton University, Ottawa, Ontario, Canada}
\email{vishalbhatoy@cmail.carleton.ca}

\author{Colin Ingalls}
\address{Department of Mathematics, Carleton University, Ottawa, Ontario, Canada}
\email{coliningalls@cunet.carleton.ca}

\begin{document}
\begin{abstract}
We compute the superpotentials for known families of Koszul Artin-Schelter regular algebras of dimension four using Magma, and apply Schur-Weyl duality from representation theory to determine the relevant invariants. Through the Borel-Weil theorem, we interpret these invariants as sections of line bundles over partial flag varieties, resulting in geometric invariants that, in some cases, correspond to K3 surfaces. We compute discrete invariants of these geometric invariants and use them to distinguish algebras.
\end{abstract}
\maketitle
\newtheorem{theorem}{Theorem}[section]  
\newtheorem{lemma}[theorem]{Lemma}       
\newtheorem{proposition}[theorem]{Proposition}  
\newtheorem{corollary}[theorem]{Corollary}      
\theoremstyle{definition}
\newtheorem{definition}[theorem]{Definition}    
\newtheorem{remark}[theorem]{Remark}            
\newtheorem{example}[theorem]{Example}           
\newtheorem{question}[theorem]{Question}         
\newtheorem{custom_def}[theorem]{} 
\newtheorem*{mainthm}{Theorem}

\section{Introduction}

Artin-Schelter (AS) regular algebras are an important class of graded algebras that can be thought of as noncommutative versions of polynomial rings. They were introduced by Artin and Schelter~\cite{Artin2007SomeAA, MR1128218, ARTIN1987171}, who classified all AS-regular algebras of global dimension up to three, assuming the generators have degree one. Since then, researchers have studied AS-regular algebras in higher dimensions, focusing on their structure, modules, and homological properties~\cite{bellamy2016noncommutative, rogalski2023artinschelter, Shelton2001OnKA, 1996_Tate, 1993}.

Dubois-Violette~\cite{duboisviolette2014multilinearformsgradedalgebras}, along with Bocklandt, Schedler, and Wemyss~\cite{BOCKLANDT20101501}, showed that every three-dimensional quadratic AS-regular algebra \(A\) is isomorphic to an \(i\)-th derivation quotient algebra \(D(w,i)\), where \(w\) is a twisted superpotential. Mori and Smith~\cite{Mori_Izuru_Smith} further proved that such a superpotential \(w\) is unique up to scalar multiplication.

We investigate the classification of 4-dimensional Koszul AS-regular algebras via their defining superpotentials, which are tensors in \( V^{\otimes 4} \), where \( \dim_{\mathbb{K}} V = 4 \). The classification problem reduces to identifying which such tensors define AS-regular algebras. Using a dataset of 77 known families of 4-dimensional regular algebras, we computed explicit superpotentials for each case. We decomposed the tensor space \( V^{\otimes 4} \) into irreducible \( \mathrm{GL}(V) \)-representations via Schur-Weyl duality, obtaining multiplicity spaces indexed by partitions of 4. Applying the Borel-Weil theorem, we interpreted these multiplicity spaces as spaces of global sections of line bundles on partial flag varieties. The relations of each family are given in~\cite[Appendix]{bhatoy-colin-felix-ravali-components}.

Each superpotential \( w \in V^{\otimes 4} \) yields five projective invariants, corresponding to linear systems of hypersurfaces in associated flag varieties. The simplest cases recover classical invariants such as the symmetrized quartic in \( \mathbb{P}^3 \) and the determinant. In several families, we found irreducible quartic surfaces that are singular along skew lines, while in others, the surfaces are reducible unions of a plane and a cubic surface. Moreover, we sometimes obtain K3 surfaces. We computed invariants for all 77 families using Magma~\cite{magma}. All the codes used to compute these invariants are available at the GitHub repository~\cite{vishal2025}.

This work is related to our paper~\cite{bhatoy-colin-felix-ravali-components}, where we studied the algebraic stack $\mathcal{A}_4$ and found 45 components by studying the families of AS-regular algebras. We computed discrete invariants of the associated linear systems to distinguish these components. In particular, we computed the dimensions and degrees of the primary components of the scheme, its reduced subscheme, its singular subscheme, and the reduced subscheme of the singular subscheme. We also computed other invariants, such as the line scheme, point scheme, and Nakayama automorphism, but they did not provide much distinction and are not discussed in this paper.

The invariants we used to classify 4-dimensional Koszul AS-regular algebras, however, are not twist-invariant. For example, consider the algebra \( A = \mathbb{K}\langle x,y\rangle / (yx - xy) \), 
and twist it by a scalar \( q \in \mathbb{K}^{\times} \) to obtain 
\( A_q = \mathbb{K}\langle x,y\rangle / (yx - q xy), \) which is a twist of $A_1.$ Decomposing the relations into symmetric and alternating parts, 
\( R \subseteq V \otimes V = \mathrm{Sym}^2 V \oplus \bigwedge^2 V \), we see that \( \mathrm{Sym}^2 V = 0 \) for \( q = 1 \), while \( \mathrm{Sym}^2 V \neq 0 \) for \( q \neq 1 \). 

Section~2 provides the necessary background and definitions. We recall connected graded algebras and define Artin-Schelter (AS) regular algebras, listing the possible minimal projective resolution types in global dimension~4. We then define superpotentials and Koszul algebras. We present Theorem~\ref{Aisomorphic-Dw2}, which shows that every 4-dimensional Koszul AS-regular algebra is isomorphic to a derivation-quotient algebra \(D(w,2)\).

In Section 3, we review Schur-Weyl duality and the Borel-Weil theorem.
Using these results, we define the geometric locus $X_\lambda(w)$ associated to a highest weight $\lambda$, which is the base locus of the corresponding linear system on a flag variety $G/B.$ We give an example for partitions $\lambda \vdash 4$. For each case, we describe the partial flag varieties, the associated line bundles, 
and the spaces of global sections that determine the geometric locus.

In Section~4, we study an example Algebra H and compute its invariants for the partition \(\lambda = (4)\). We also discuss the concept of absolute irreducibility, explaining how each primary component becomes absolute irreducible over a finite field extension. 
We then discuss in detail the computations for the partition $\lambda = (2,2)$.

In Section~5, we perform computations for the 77 known families of 4-dimensional Koszul AS-regular algebras and calculate their invariants for all partitions $\lambda \vdash 4.$

In Section~6, we present the main result, Theorem~\ref{maintheorem-inv}. Using this theorem, we partition the 45 components of the algebraic stack \(\mathcal{A}_4\) into 39 equivalence classes, or boxes, where each class consists of algebras that share the same discrete invariants coming from the projective schemes \(X_{400}, X_{000}, X_{020}, X_{101}, X_{210}\). The theorem shows that if \(A\) and \(B\) are generic families in different boxes, then \(A \not\simeq B\), so components in different equivalence classes are distinct.

\section{Background and Definitions}
Let $\mathbb{K}$ be an algebraically closed field of characteristic zero.
We will say $A$ is connected graded if $A$ is graded $A = \bigoplus_{i \geq 0} A_i$ with $A_0=\mathbb{K}.$
\begin{definition}
\textbf{Artin-Schelter regular algebra} \cite{ARTIN1987171}:
 An Artin-Schelter regular algebra (AS-regular algebra) is a connected graded algebra $A$ over a field $\mathbb{K}$ generated
in degree one, such that the following conditions hold:
\begin{enumerate}
    \item $A$ has finite global dimension $d$.
    \item \textnormal{\( A \) has polynomial growth, i.e., \( f(n) = \dim_\mathbb{K} A_n \) is bounded above by a polynomial function of \( n \).}
    \item \( A \) is Gorenstein, i.e.,
\[
\Ext^{i}_{A}(\mathbb{K}, A) \cong 
\begin{cases} 
0 & \text{if } i \neq d, \\ 
 \mathbb{K}(l) & \text{if } i = d 
\end{cases}
\]
for some shift of grading \( l \in \mathbb{Z} \). This integer \( l \) is called the Gorenstein parameter.
\end{enumerate}
\end{definition}

By the dimension of an AS-regular algebra, we will mean the integer $d$ appearing above. If $A$ is AS-regular, then the trivial left $A$-module $\mathbb{K}$ has a minimal free resolution of the form
\[ 0 \rightarrow P_d \rightarrow \dots \rightarrow P_1 \rightarrow P_0 \rightarrow \mathbb{K} \rightarrow 0, \]
where $P_w = \bigoplus_{s=1}^{n_w} A(-i_{w,s})$ for some positive integers $n_w$ and $i_{w,s}$. The Gorenstein condition (3) implies that the above free resolution is symmetric in the sense that the dual complex of the above sequence is a free resolution of the trivial right $A$-module (after a degree shift). As a consequence, we have $P_0 = A$, $P_d = A(-l)$, $n_w = n_{d-w}$ and $i_{w,s} + i_{d-w,n_w-s+1} = l$ for all $w, s$.

Under the natural hypothesis that $A$ is a Noetherian domain and $
\dim A = 4$, there are three possible resolution types  \cite[Prop.~1.4]{Lu_2007}:
\begin{align*}\label{eq:type3}
0 \rightarrow A(-4) \rightarrow A(-3)^4 \rightarrow A(-2)^6 \rightarrow A(-1)^4 \rightarrow A \rightarrow \mathbb{K} \rightarrow 0
\end{align*}
\begin{align*}
0 \rightarrow A(-5) \rightarrow A(-4)^3 \rightarrow A(-3)^2 \oplus A(-2)^2 \rightarrow A(-1)^3 \rightarrow A \rightarrow \mathbb{K} \rightarrow 0   
\end{align*}
\begin{align*}
0 \rightarrow A(-7) \rightarrow A(-6)^2 \rightarrow A(-4) \oplus A(-3) \rightarrow A(-1)^2 \rightarrow A \rightarrow \mathbb{K} \rightarrow 0    
\end{align*}
with Hilbert series, respectively, 
$$ \frac{1}{(1 - t)^4}, \frac{1}{(1 - t)^3 (1 - t^2)}, \frac{1}{(1 - t)^2 (1 - t^2) (1 - t^3)}.$$

Suggested by the form of the first resolution above, we say such an algebra is of type \textit{(14641)}. In this paper we mainly deal with algebras of type \textit{(14641)}. An algebra of type \textit{(14641)} is Koszul \cite[Theorem 0.1]{MR2529094}. So we will restrict attention to the situation where $A$ is generated in degree 1, and has defining relations in degree 2. 

\begin{definition}\label{def:superpotential}
\textbf{Superpotential} \cite[Definition 2.5]{Mori_Izuru_Smith}:
Let $m \in \mathbb{N}^+$ and let $V$ be a finite-dimensional $\mathbb{K}$-vector space.  
Define a linear map $\varphi : V^{\otimes m} \to V^{\otimes m}$ by  
\[
\varphi(v_1 \otimes v_2 \otimes \cdots \otimes v_{m}) := v_m \otimes v_1 \otimes \cdots \otimes v_{m-1}.
\]
\begin{enumerate}
    \item An element $w \in V^{\otimes m}$ is called a superpotential if $\varphi(w) = w$.
    
    \item An element $w \in V^{\otimes m}$ is called a twisted superpotential if 
    \[
    (\mu \otimes \mathrm{id}^{\otimes (m-1)})\varphi(w) = w
    \]
    for some $\mu \in \mathrm{GL}(V)$. This $\mu$ is called the Nakayama automorphism.
\end{enumerate}  
\end{definition}

\begin{definition}
\textbf{Tensor algebra}\label{def:derivationquotient-alg}:
Let $V$ be a finite-dimensional $\mathbb{K}$-vector space with basis $\{ x_1, \dots, x_n \}$. The tensor algebra of $V$ is defined by
$$ T(V) := \bigoplus_{m \geq 0} V^{\otimes m},$$
where $V^{\otimes 0} := \mathbb{K}$.
Given an element $w \in V^{\otimes m}$ with $m \geq 1$, there exist unique elements $w_j \in V^{\otimes (m-1)}$ such that
\[
w = \sum_{j=1}^n x_j \otimes w_j.
\]
The element $\partial_{x_j} w := w_j$ is called the left partial derivative of $w$ with respect to $x_j$.

More generally, for any positive integer $i$, the $i$-th order left partial derivatives of $w$ are defined recursively. For any sequence of indices $(j_1, \dots, j_i) \in \{1, \dots, n\}^i$, we define
\[
\partial_{x_{j_1}} \cdots \partial_{x_{j_i}} w := \partial_{x_{j_1}} \left( \partial_{x_{j_2}} \cdots \partial_{x_{j_i}} w \right),
\]
where each partial derivative acts on the leftmost tensor factor.

Let $\partial^i w$ denote the set of all such $i$-th order left partial derivatives:
\[
\partial^i w := \left\{ \partial_{x_{j_1}} \cdots \partial_{x_{j_i}} w \;\middle|\; 1 \leq j_1, \dots, j_i \leq n \right\}.
\]

The $i$-th derivation quotient algebra associated to $w$ is defined as
\[
D(w, i) := T(V) \big/ \left( \partial^i w \right),
\]
i.e., the quotient of the tensor algebra by the ideal generated by all $i$-th order left partial derivatives of $w$.

We focus on the case where $\dim V = 4$ and $i = 2$. In this setting, $D(w,2)$ is generated by the $16$ second-order left partial derivatives of $w$. However, only $6$ of them are distinct.

The Nakayama automorphism $\mu$ from Definition~\ref{def:superpotential} is in $\mathrm{Aut}(D(w,2))$~\cite[Thm.~1.10]{Mori_Izuru_Smith}.
\end{definition}

\begin{example}\label{superpolyring}
Let \( V \) be a 4-dimensional \( \mathbb{K} \)-vector space with basis \( \{x_1, x_2, x_3, x_4\} \). Consider the element \( w \in V^{\otimes 4} \) defined by
\[
\begin{aligned}
w :=\ & -x_1x_2x_3x_4 + x_1x_2x_4x_3 + x_1x_3x_2x_4 - x_1x_3x_4x_2 - x_1x_4x_2x_3 + x_1x_4x_3x_2 \\
& + x_2x_1x_3x_4 - x_2x_1x_4x_3 - x_2x_3x_1x_4 + x_2x_3x_4x_1 + x_2x_4x_1x_3 - x_2x_4x_3x_1 \\
& - x_3x_1x_2x_4 + x_3x_1x_4x_2 + x_3x_2x_1x_4 - x_3x_2x_4x_1 - x_3x_4x_1x_2 + x_3x_4x_2x_1 \\
& + x_4x_1x_2x_3 - x_4x_1x_3x_2 - x_4x_2x_1x_3 + x_4x_2x_3x_1 + x_4x_3x_1x_2 - x_4x_3x_2x_1.
\end{aligned}
\]

Since \( \varphi(w) = w \), it follows that \( w \) is a superpotential.

We now express \( w \) in the form
\[
\begin{aligned}
w =\ & x_1x_2(-x_3x_4 + x_4x_3) + x_1x_3(x_2x_4 - x_4x_2) + x_1x_4(-x_2x_3 + x_3x_2) \\
& + x_2x_1(x_3x_4 - x_4x_3) + x_2x_3(-x_1x_4 + x_4x_1) + x_2x_4(x_1x_3 - x_3x_1) \\
& + x_3x_1(-x_2x_4 + x_4x_2) + x_3x_2(x_1x_4 - x_4x_1) + x_3x_4(-x_1x_2 + x_2x_1) \\
& + x_4x_1(x_2x_3 - x_3x_2) + x_4x_2(-x_1x_3 + x_3x_1) + x_4x_3(x_1x_2 - x_2x_1).
\end{aligned}
\]

Alternatively, we can write this as
\[
\begin{aligned}
w &= x_1x_2\, \partial^2_{x_1x_2} w + x_1x_3\, \partial^2_{x_1x_3} w + x_1x_4\, \partial^2_{x_1x_4} w \\
&\quad + x_2x_1\, \partial^2_{x_2x_1} w + x_2x_3\, \partial^2_{x_2x_3} w + x_2x_4\, \partial^2_{x_2x_4} w \\
&\quad + x_3x_1\, \partial^2_{x_3x_1} w + x_3x_2\, \partial^2_{x_3x_2} w + x_3x_4\, \partial^2_{x_3x_4} w \\
&\quad + x_4x_1\, \partial^2_{x_4x_1} w + x_4x_2\, \partial^2_{x_4x_2} w + x_4x_3\, \partial^2_{x_4x_3} w,
\end{aligned}
\]
where
\[
\begin{array}{rlrl}
\partial^2_{x_1x_2} w &= -x_3x_4 + x_4x_3, &
\partial^2_{x_2x_1} w &= x_3x_4 - x_4x_3, \\[3pt]
\partial^2_{x_1x_3} w &= x_2x_4 - x_4x_2, &
\partial^2_{x_3x_1} w &= -x_2x_4 + x_4x_2, \\[3pt]
\partial^2_{x_1x_4} w &= -x_2x_3 + x_3x_2, &
\partial^2_{x_4x_1} w &= x_2x_3 - x_3x_2, \\[3pt]
\partial^2_{x_2x_3} w &= -x_1x_4 + x_4x_1, &
\partial^2_{x_3x_2} w &= x_1x_4 - x_4x_1, \\[3pt]
\partial^2_{x_2x_4} w &= x_1x_3 - x_3x_1, &
\partial^2_{x_4x_2} w &= -x_1x_3 + x_3x_1, \\[3pt]
\partial^2_{x_3x_4} w &= -x_1x_2 + x_2x_1, &
\partial^2_{x_4x_3} w &= x_1x_2 - x_2x_1,
\end{array}
\]
and 
\[
\partial^2_{x_i x_i} w = 0 \quad \text{for all } i.
\]

Also note that 
\[
\begin{aligned}
\partial^2_{x_j x_i} w &= - \partial^2_{x_i x_j} w, \quad \text{for } i \neq j. \\    
\end{aligned}
\]

The derivation-quotient algebra is:
\[
D(w, 2) = \mathbb{K}\langle x_1, x_2, x_3, x_4 \rangle \big/ \left( \partial^2_{x_i x_j} w \mid 1 \leq i < j \leq 4 \right).
\]
Hence,
\[
D(w, 2) \cong \mathbb{K}[x_1, x_2, x_3, x_4],
\]
the commutative polynomial ring in four variables.
\end{example}

\begin{example}\label{superskewpolyring}
Consider the skew polynomial ring in four variables, defined by
\[
A = \mathbb{K} \langle x_1, x_2, x_3, x_4 \rangle \big/ (x_i x_j - q_{ij} x_j x_i),
\]
for parameters \( q_{ij} \in \mathbb{K}^\times \) satisfying \(1 \leq i < j \leq 4\).

The corresponding superpotential is
\[
w = \sum_{\sigma \in S_4} (-1)^\sigma \alpha_\sigma \, x_{\sigma(1)} \ldots x_{\sigma(4)},
\]
where \( \alpha_\sigma \in \mathbb{K} \) are scalars depending on the permutation \( \sigma \in S_4 \).

The second-order cyclic partial derivatives of \( w \) are given by
\[
\partial^2_{x_i x_j}w = -\alpha_{jikl} x_k x_l + \alpha_{jilk} x_l x_k,
\quad \text{and} \quad
\partial^2_{x_i x_i}w = 0,
\]
where \( i, j, k, l \in \{1,2,3,4\} \) are all distinct.

These coefficients satisfy the conditions
\[
\alpha_{iklj} \alpha_{kijl} = \alpha_{kilj} \alpha_{ikjl},
\quad \text{and} \quad
q_{ij} = \frac{\alpha_{klij}}{\alpha_{klji}}.
\]
\end{example}

\begin{theorem}\cite[Theorem 11]{duboisviolette2014multilinearformsgradedalgebras}\label{Aisomorphic-Dw2}:
For every 4-dimensional Koszul AS-regular algebra $A$, there exists a unique twisted superpotential $ w \in V^{\otimes 4}$ such that
    \[
    A \simeq D(w, 2).
    \]
\end{theorem}

\begin{definition}
\textbf{Koszul Algebras}:
Let \( A = T(V)/\langle R \rangle \) be a connected \(\mathbb{N}\)-graded quadratic algebra over a field \( \mathbb{K} \), where \( R \subseteq V \otimes V \). Its quadratic dual is
\[
A^! = T(V^*) / \langle R^\perp \rangle,
\]
where \( R^\perp \subseteq V^* \otimes V^* \) is the orthogonal complement of \( R \).

Let \( h_A(t) \) and \( h_{A^!}(t) \) be the Hilbert series of \( A \) and \( A^! \), respectively.
If \( A \) is Koszul, then the minimal graded projective resolution of the trivial \( A \)-module \( \mathbb{K} \) is linear, and the Koszul duality identity holds \cite{Positselskii1995} (see also \cite[Thm.~5.9]{MR1388568}):
\[
h_A(t)\, h_{A^!}(-t) = 1.
\]
    
\end{definition}

\begin{definition}
\textbf{Frobenius Algebra}:
A finite-dimensional graded associative \(\mathbb{K}\)-algebra \(A^!\) is called Frobenius of degree \(n\) if there exists a nondegenerate bilinear pairing
\[
\langle -, - \rangle : A^! \otimes A^! \to \mathbb{K}
\]
such that
\[
\langle a, b \cdot c \rangle = \langle a \cdot b, c \rangle \quad \text{for all } a, b, c \in A^!.
\]
\end{definition}

\begin{theorem}{\cite[Corollary D]{lu2007koszul}}\label{thmfrobn}
Let \(A\) be a connected graded algebra. Then \(A\) is AS-regular if and only if the Ext-algebra \(E := \bigoplus_{i= 0}^{\infty} \Ext^i_A(\mathbb{K}, \mathbb{K})\) is a Frobenius algebra.
\end{theorem}

\begin{corollary}{\cite[Thm.~4.3 and Prop.~5.10]{MR1388568}}
If \(A\) is a Koszul algebra, then its Ext-algebra \(E\) is isomorphic to the dual \(A^!\). Hence, for Koszul algebras, Theorem~\ref{thmfrobn} implies that \(A\) is AS-regular if and only if \(A^!\) is Frobenius.
\end{corollary}

Assume \( A \) is a Koszul AS-regular algebra of global dimension \( n \). Then its Koszul dual \( A^{!} \) is a finite-dimensional Frobenius algebra of length \( n \), with \[
\dim_\mathbb{K} A^{!}_n = 1. \]
Therefore, \( A^{!}_n \) is a one-dimensional $\mathbb{K}$-vector space. Its dual $A_n^{!*} := (A_n^{!})^*$ embeds into \( V^{\otimes n} \), since $ A_n^{!*} \subseteq \big( (V^*)^{\otimes n} \big)^* \cong V^{\otimes n}.$

The superpotential \( w \) is defined as a generator (unique up to scalar) of this space:
\[
w \in A_n^{!*}.
\]
Moreover, this space can be explicitly realized as the intersection
\[
A_n^{!*} = \bigcap_{i=0}^{n - 2} V^{\otimes i} \otimes R \otimes V^{\otimes (n - 2 - i)} \subseteq V^{\otimes n},
\]
as proved in \cite[Lemma 2.11]{Mori_Izuru_Smith}.

In particular, for the Koszul algebra with global dimension \( n = 4 \), this becomes
$$ w \in A_4^{!*} = \bigcap_{i=0}^{2} V^{\otimes i} \otimes R \otimes V^{\otimes (2 - i)} \subseteq V^{\otimes 4}.$$
We develop Magma code that computes the superpotential for all Koszul AS-regular algebras of dimension four~\cite{vishal2025}. 

\begin{example}\label{example:algebraH}
Consider the Algebra H. The relations are given by
\[
\begin{alignedat}{2}
& x_2x_1 - x_1x_2 - x_1^2, &\quad & x_4x_3 + x_3x_4, \\
& x_3x_1 + h(-x_1x_4), &\quad & x_3x_2 + h(-fx_1x_4 - x_2x_4), \\
& x_4x_1 + h(-x_1x_3), &\quad & x_4x_2 + h(-fx_1x_3 - x_2x_3)
\end{alignedat}
\]
where \( f, h \in \mathbb{K} \).

Using the Magma code~\cite{vishal2025}, we compute the superpotential $w$ as follows:

$  h^4  x_1^2  x_3  x_4
+ h^4  x_1^2  x_4  x_3
+ h^4  x_1  x_2  x_3  x_4
+ h^4  x_1  x_2  x_4  x_3
+ (h^3 f - h^3)  x_1  x_3  x_1  x_3
- h^3  x_1  x_3  x_2  x_3
+ (-2 h^2 f + h^2)  x_1  x_3  x_4  x_1
+ h^2  x_1  x_3  x_4  x_2
+ (h^3 f - h^3)  x_1  x_4  x_1  x_4
- h^3  x_1  x_4  x_2  x_4
+ (-2 h^2 f + h^2)  x_1  x_4  x_3  x_1
+ h^2  x_1  x_4  x_3  x_2
- h^4  x_2  x_1  x_3  x_4
- h^4  x_2  x_1  x_4  x_3
+ h^3  x_2  x_3  x_1  x_3
- h^2  x_2  x_3  x_4  x_1
+ h^3  x_2  x_4  x_1  x_4
- h^2  x_2  x_4  x_3  x_1
+ h^2  x_3  x_1^2  x_4
+ h^2  x_3  x_1  x_2  x_4
+ (h f - h)  x_3  x_1  x_3  x_1
- h  x_3  x_1  x_3  x_2
- h^2  x_3  x_2  x_1  x_4
+ h  x_3  x_2  x_3  x_1
+ x_3  x_4  x_1^2
+ x_3  x_4  x_1  x_2
- x_3  x_4  x_2  x_1
+ h^2  x_4  x_1^2  x_3
+ h^2  x_4  x_1  x_2  x_3
+ (h f - h)  x_4  x_1  x_4  x_1
- h  x_4  x_1  x_4  x_2
- h^2  x_4  x_2  x_1  x_3
+ h  x_4  x_2  x_4  x_1
+ x_4  x_3  x_1^2
+ x_4  x_3  x_1  x_2
- x_4  x_3  x_2  x_1 .$
\end{example}

\section{Representation Theory}
In this section, we present two important theorems in representation theory: \textbf{Schur-Weyl duality} and the \textbf{Borel-Weil theorem}. Schur-Weyl duality reveals a deep connection between the irreducible representations of the symmetric group and the irreducible algebraic representations of the general linear group over a finite-dimensional vector space defined over a field $\mathbb{K}$. The Borel-Weil theorem provides a geometric realization of irreducible representations of linear algebraic groups by identifying them with spaces of global sections of line bundles over partial flag varieties.

\begin{proposition}~\cite[Theorem 3.1]{etingof2011introductionrepresentationtheory}\label{prop:621}
Let $G$ be a finite group and $\mathbb{K}$ be a field such that
$\operatorname{char}(\mathbb{K}) \nmid |G|$.
Then $\mathbb{K}[G]$ is semisimple and
\[
\mathbb{K}[G] \cong \bigoplus_i \End_{\mathbb{K}}(V_i)
\]
as $\mathbb{K}$-algebras,
where the direct sum is taken over all distinct irreducible representations \(V_i\) of \(G\).
\end{proposition}
Let $c_\lambda$ be the Young symmetrizer corresponding to the partition $\lambda \vdash n$, as in \cite[page 58]{etingof2011introductionrepresentationtheory}.

\begin{theorem}~\cite[Theorem 4.36]{etingof2011introductionrepresentationtheory}\label{thm:641ksn}
The subspace
\[
U_\lambda = \mathbb{K}[S_n] \, c_\lambda
\]
obtained by left multiplication of $c_\lambda$ is an irreducible representation of $S_n$. Moreover, every irreducible representation of $S_n$ is isomorphic to some $U_\lambda$ for a unique partition $\lambda$ of $n$.
\end{theorem}
Since \(\mathbb{K}[S_n]\) is a semisimple algebra, it decomposes as a direct sum of matrix algebras corresponding to the irreducible representations of \(S_n\), as in Proposition~\ref{prop:621}:
$$ \mathbb{K}[S_n] \cong \bigoplus_{\lambda \vdash n} \operatorname{End}(U_\lambda),$$
where the sum ranges over all partitions \(\lambda \vdash n\), and each \(U_\lambda\) is an irreducible representation of \(S_n\) over \(\mathbb{K}\).

This decomposition corresponds to a decomposition of the identity element into a sum of orthogonal central idempotents:
\[
1 = \sum_{\lambda \vdash n} e_\lambda,
\]
where each \(e_\lambda \in \mathbb{K}[S_n]\) is a central idempotent satisfying $\mathbb{K}[S_n] e_\lambda \cong \operatorname{End}(U_\lambda).$

Each idempotent \( e_\lambda \in \mathbb{K}[S_4] \) projects onto the \( \lambda \)-isotypic component and is given by
\begin{equation*}~\label{eqchformula}
e_\lambda = \frac{\dim U_{\lambda}}{4!} \sum_{g \in S_4} \chi_\lambda(g) \cdot g,
\end{equation*}
where \( \chi_\lambda(g) \) denotes the character of \( g \) in the irreducible representation labeled by \( \lambda \). The values \( \chi_\lambda(g) \) are taken from the character table of \( S_4 \)~\cite[page 19]{fulton1991representation}.

In particular, each central idempotent \(e_\lambda\) acts as a projection onto the \(\lambda\)-isotypic component in \(\mathbb{K}[S_n]\)-module \(V\). More precisely, the image of \(e_\lambda\) in \(V\) is given by
\[
e_\lambda V = \operatorname{Hom}_{S_n}(U_\lambda, V) \otimes U_\lambda,
\]
which corresponds to the multiplicity space of the irreducible representation \(U_\lambda\) within \(V\).

To further analyze the superpotential \( w \in V^{\otimes 4} \), we consider the natural action of the symmetric group \( S_4 \) on \( V^{\otimes 4} \) by permuting tensor positions. Since this action commutes with the diagonal \( \mathrm{GL}(V) \)-action~\cite[page 76]{fulton1991representation}, the space decomposes into isotypic components corresponding to irreducible representations of \( S_4 \):
$$ w = \sum_{\lambda \vdash 4} w_\lambda, \quad \text{where} \quad w_\lambda = e_\lambda \cdot w.$$
To demonstrate the decomposition of the superpotential $w$ into isotypic components, we present some examples below.
All computations are done using Magma~\cite{vishal2025}.

\begin{example}
For the polynomial ring \( A = \mathbb{K}[x_1,x_2,x_3,x_4] \), the superpotential \( w \in V^{\otimes 4} \), as in Example~\ref{superpolyring}, lies entirely in the alternating component of \( V^{\otimes 4} \).

We project \( w \) onto the isotypic components of \( V^{\otimes 4} \) corresponding to the irreducible representations of \( S_4 \). Denoting \( w_\lambda := e_\lambda \cdot w \), we find:
\[
w = w_{(1,1,1,1)} \qquad \text{and} \qquad w_\lambda = 0 \quad \text{for all } \lambda \ne (1,1,1,1).
\]
That is, the superpotential decomposes as
\[
w = e_{(4)} \cdot 0 + e_{(3,1)} \cdot 0 + e_{(2,2)} \cdot 0 + e_{(2,1,1)} \cdot 0 + e_{(1,1,1,1)} \cdot w.
\]
\end{example}

\begin{example}
Consider the Clifford algebra~\cite[Section 1.1]{Cassidy2010GeneralizationsOG}. In this case all components vanish except the trivial component:
\[
w = w_{(4)} \qquad \text{and} \qquad w_\lambda = 0 \quad \text{for all } \lambda \ne (4).
\]
That is, the superpotential decomposes as
\[
w = e_{(4)} \cdot w + e_{(3,1)} \cdot 0 + e_{(2,2)} \cdot 0 + e_{(2,1,1)} \cdot 0 + e_{(1,1,1,1)} \cdot 0.
\]
\end{example}

\begin{theorem}
[\textbf{Schur-Weyl Duality}, {\cite[Corollary 4.59]{etingof2011introductionrepresentationtheory}}]
\label{schurweylduality}
Let \( V \) be a finite-dimensional vector space over a field \( \mathbb{K} \). Then, as a representation of \( \GL(V) \times S_n \), there is a decomposition
\[
V^{\otimes n} \simeq \bigoplus_{|\lambda| = n} \mathbb{S}_\lambda V \otimes U_\lambda,
\]
where
\begin{itemize}
    \item \( \lambda \) ranges over all partitions of \( n \),
    \item \( U_\lambda \) is the irreducible representation of the symmetric group \( S_n \) corresponding to \( \lambda \),
    \item \( \mathbb{S}_\lambda V := \Hom_{S_n}(U_\lambda, V^{\otimes n}) \) is the multiplicity space, carrying an irreducible representation of \( \GL(V) \).
\end{itemize}  
\end{theorem}

Moreover, \( \mathbb{S}_\lambda V = 0 \) whenever the partition \( \lambda \) has more than \( d \) parts (i.e., when the \((d+1)\)-th part of \( \lambda \) is nonzero), where $d = \dim V.$

\begin{remark}
Schur-Weyl duality does not require $\dim V = n$ in general; in our case, $\dim V = n = 4$, so all partitions of 4 are allowed and no Schur functor vanishes.
\end{remark}

\begin{example}
The partitions of \(2\) are \((2)\) and \((1,1)\). By Schur-Weyl duality, we have
$V^{\otimes 2} \simeq \mathrm{Sym}^2 V \oplus \bigwedge^2 V.$

\end{example}

\begin{example}
The partitions of \(3\) are \((3)\), \((2,1)\), and \((1,1,1)\). By Schur-Weyl duality,
$ V^{\otimes 3} \simeq \mathrm{Sym}^3 V \oplus (\mathbb{S}_{(2,1)} V)^{\oplus 2} \oplus \bigwedge^3 V.$
The representation \(\mathbb{S}_{(2,1)} V\) does not admit as simple a description as the symmetric or alternating powers. It can be shown that~\cite[page 76]{fulton1991representation}:
\[
\mathbb{S}_{(2,1)} V \simeq V^{\otimes 3} c_{(2,1)} \cong \ker \bigl( V \otimes \bigwedge^2 V \xrightarrow{\varphi} \bigwedge^3 V \bigr),
\]
where the map \(\varphi\) is the alternating map
\[
\varphi : V \otimes \bigwedge^2 V \to \bigwedge^3 V, \quad \varphi(v_1 \otimes (v_2 \bigwedge v_3)) = v_1 \bigwedge v_2 \bigwedge v_3.
\]
\end{example}

\begin{example}
The partitions of \(4\) are:
\[
(4), \quad (3,1), \quad (2,2), \quad (2,1,1), \quad (1,1,1,1).
\]

Schur-Weyl duality yields the decomposition:
\[
\begin{aligned}
V^{\otimes 4} \simeq &\ \mathbb{S}_{(4)} V \otimes U_{(4)} \oplus \mathbb{S}_{(3,1)} V \otimes U_{(3,1)} \oplus \mathbb{S}_{(2,2)} V \otimes U_{(2,2)} \\
&\oplus \mathbb{S}_{(2,1,1)} V \otimes U_{(2,1,1)} \oplus \mathbb{S}_{(1,1,1,1)} V \otimes U_{(1,1,1,1)} \\
 \simeq &\ \mathrm{Sym^4} V  \oplus (\mathbb{S}_{(3,1)} V \otimes U_{(3,1)}) \oplus (\mathbb{S}_{(2,2)} V \otimes U_{(2,2)}) \\
&\oplus (\mathbb{S}_{(2,1,1)} V \otimes U_{(2,1,1)}) \oplus 
\bigwedge^4 V. 
\end{aligned}
\]
\end{example}

\begin{definition}\label{def:partial-flag-variety}
(\textbf{Partial Flag Variety}, {\cite[Definition 9.5.1]{lakshmibai2009flag}})
Let \( d = (d_1 < d_2 < \cdots < d_r) \) be a strictly increasing sequence of integers with \( d_r < n \). The \textbf{partial flag variety}
\[
F(d; n) = F(d_1, d_2, \ldots, d_r; n)
\]
is defined as the set of \( r \)-tuples of subspaces
\[
(W_{d_1}, W_{d_2}, \ldots, W_{d_r})
\]
satisfying
\[
W_{d_1} \subset W_{d_2} \subset \cdots \subset W_{d_r} \subset \mathbb{K}^n,
\]
with \(\dim W_{d_i} = d_i\) for each \( i \). This variety embeds into the product of Grassmannians
\[
F(d; n) \subset \Gr(d_1, n) \times \Gr(d_2, n) \times \cdots \times \Gr(d_r, n).
\]
The \textbf{full} or \textbf{complete flag variety} \( F(n) \) corresponds to the special case
\[
r = n - 1 \quad \text{and} \quad d = (1 < 2 < \cdots < n - 1),
\]
whose points are called complete flags.
\end{definition}
Note that, by \cite[Example 9.5.2]{lakshmibai2009flag}, the partial flag variety 
$F(d;n)$ is a closed subvariety of 
$\Gr(d_1,n) \times \cdots \times \Gr(d_r,n)$; hence it is projective.

\begin{proposition}\cite[Proposition 2.1]{friedman2025grassmannflagvarietieslinear}
Let \( d = (d_1 < d_2 < \cdots < d_r) \) be a strictly increasing tuple and define \( d_0 := 0 \). Then the partial flag variety \( F(d; n) \) is irreducible and has dimension
\[
\dim F(d; n) = \sum_{i=1}^r (d_i - d_{i-1})(n - d_i).
\]
In particular, the full flag variety \( F(n) = F(1, 2, \ldots, n-1; n) \) has dimension
\[
\dim F(n) = \binom{n}{2}.
\]
\end{proposition}

\begin{definition}
Let \( V \) be a finite-dimensional vector space over a field \( \mathbb{K} \), and fix a basis of \( V \). Then the subgroup \( B \subset \mathrm{GL}(V) \) consisting of all invertible upper-triangular matrices with respect to this basis is called a \textbf{Borel subgroup}.
The subgroup \( T \subset \mathrm{GL}(V) \) consisting of all invertible diagonal matrices with respect to the same basis is called a \textbf{Maximal torus}.
\end{definition}

\begin{definition}
\textbf{Weight vector}:
A nonzero vector \( v \in V \) is called a weight vector if it is an eigenvector for all elements \( t \in T \); that is, for each \( t \in T \), there exists a scalar \( \chi(t) \in \mathbb{K}^\times \) such that
\[
t \cdot v = \chi(t) v.
\]
The map \( \chi : T \to \mathbb{K}^\times \) is called the weight of \( v \).
\end{definition}

\begin{definition}
\textbf{Highest weight vector}:
Let \( B \subset \GL(V) \) be a Borel subgroup containing the maximal torus \( T \subset B \). A nonzero vector \( v \in V \) is called a highest weight vector if for all \( b \in B \),
\[
b \cdot v = \mu(b) v,
\]
where \( \mu : B \to \mathbb{K}^\times \) is a group homomorphism (i.e., a character of \( B \)).

In particular, since \( T \subset B \), the vector \( v \) is a weight vector with weight given by the restriction of \( \mu \) to \( T \). The map \( \mu \) is called the highest weight associated to \( v \).
\end{definition}

\begin{definition}
\textbf{Dominant weight}:
A weight 
\[
\lambda = (\lambda_1, \lambda_2, \ldots, \lambda_n) \in \mathbb{Z}^n
\]
is called dominant (w.r.t the Borel subgroup of $\GL_{n}$) if
\[
\lambda_1 \geq \lambda_2 \geq \ldots \geq \lambda_n.
\]
\end{definition}

\begin{definition}
\textbf{Difference tuple} \(\mathbf{d_\lambda}\):  
Let \(\lambda = (\lambda_1, \lambda_2, \ldots, \lambda_k)\) be a dominant weight.
We extend \(\lambda\) by zeros to length \( n \):
\[
\lambda = (\lambda_1, \lambda_2, \ldots, \lambda_k, 0, \ldots, 0).
\]
Such a dominant weight \(\lambda\) corresponds to a partition of length at most \( n \). The partition can be identified with the Young diagram having rows of lengths \(\lambda_1, \lambda_2, \ldots, \lambda_n\).

Define the difference tuple \( d_\lambda \) associated to \(\lambda\) as
\[
d_\lambda = (\lambda_1 - \lambda_2, \lambda_2 - \lambda_3, \ldots, \lambda_{n-1} - \lambda_n),
\]
which is a tuple of non-negative integers since \(\lambda\) is dominant.
\end{definition}

\begin{example}\label{notations-lambda-}
We now provide examples of partitions of 4 and their corresponding difference tuples.

\begin{center}
\renewcommand{\arraystretch}{1.2}
\begin{tabular}{|c|c|c|}
\hline
Partition \( \lambda \) & Extended \( \lambda \) (length 4) & Difference tuple \( d_\lambda \) \\
\hline
\( (4) \) & \( (4,0,0,0) \) & \( (4,0,0) \) \\
\( (3,1) \) & \( (3,1,0,0) \) & \( (2,1,0) \) \\
\( (2,2) \) & \( (2,2,0,0) \) & \( (0,2,0) \) \\
\( (2,1,1) \) & \( (2,1,1,0) \) & \( (1,0,1) \) \\
\( (1,1,1,1) \) & \( (1,1,1,1) \) & \( (0,0,0) \) \\
\hline
\end{tabular}
\end{center}
\end{example}


Let $V$ be an $n$-dimensional vector space over a field $\mathbb{K}$, $G = \GL(V)$, and consider the flag variety $G/B$.

\begin{theorem}[\textbf{Borel-Weil}, {\cite[Proposition 11.2.2]{lakshmibai2009flag}}]\label{borelweilthm}

Let \( \lambda = (\lambda_1, \ldots, \lambda_k) \) be a dominant weight with difference tuple $d_\lambda.$ Then
$$\mathrm{H}^0(G/B, L_{\lambda}) \cong \mathbb{S}_{\lambda}(V),$$
where \( \mathbb{S}_\lambda(V) \) denotes the irreducible representation of \( \GL(V) \) with highest weight \( \lambda \) (equivalently, with difference tuple \( d_\lambda \)).

\end{theorem}

By Schur-Weyl duality~\ref{schurweylduality}, we have the decomposition
\[
V^{\otimes n} \cong \bigoplus_{\lambda \vdash n} \mathbb{S}_\lambda(V) \otimes U_\lambda.
\]
By Borel-Weil~\ref{borelweilthm}, each irreducible summand \( \mathbb{S}_\lambda(V) \) can be realized as the space of global sections of the line bundle $L_\lambda$ on the flag variety \( G/B \) as
\[
\mathbb{S}_\lambda(V) \cong \mathrm{H}^0(G/B, L_\lambda).
\]
Therefore, the decomposition of a superpotential \( w \in V^{\otimes n} \) can be expressed as
\[
w = \sum_{\lambda \vdash n} w_\lambda, \quad \text{with } w_\lambda \in \mathrm{H}^0(G/B, L_\lambda) \otimes U_\lambda.
\]

Each component \( w_\lambda \) corresponds to a morphism of vector spaces
\[
\widetilde{w}_\lambda : U_\lambda^* \to \mathrm{H}^0(G/B, L_\lambda),
\]
whose image is a linear subsystem
\[
\widetilde{w}_\lambda(U_\lambda^*) \subseteq \mathrm{H}^0(G/B, L_\lambda).
\]

\begin{definition}\label{def:Xlambda}
\textbf{Geometric Locus}:
We define the associated geometric locus as the base locus of this linear system:
\[
X_\lambda(w) := \mathrm{Base}\big(\widetilde{w}_\lambda(U_\lambda^*)\big) \subseteq G/B,
\]
i.e., the subset of points in \( G/B \) where all sections in \( \widetilde{w}_\lambda(U_\lambda^*) \) simultaneously vanish.
\end{definition}

\begin{proposition}[{\cite[p.~142]{Fulton_1996}}]
The Picard group \(\mathrm{Pic}(G/B)\) is the group of isomorphism classes of line bundles on \(G/B\). For the complete flag variety,
\[
\mathrm{Pic}(G/B) \cong \mathbb{Z}^{n-1},
\]
where each factor corresponds to the pullback of the hyperplane line bundle \(\mathcal{O}(1)\) from the corresponding projective space in the Pl\"ucker embedding.

For each $i = 1,\dots, n-1$, let $ \pi_i : G/B \longrightarrow \Gr(i,n)$
denote the projection to the $i$-dimensional subspace in the flag. 

Thus, any line bundle \(L_\lambda\) on \(G/B\) can be expressed uniquely as
\[
L_\lambda = \pi_1^* \mathcal{O}(d_1) \otimes \pi_2^* \mathcal{O}(d_2) \otimes \cdots \otimes \pi_{n-1}^* \mathcal{O}(d_{n-1}),
\]
where \((d_1, d_2, \dots, d_{n-1})\) is the difference tuple of the highest weight \(\lambda\). 
\end{proposition}

\begin{proposition}\label{prop:flag-grassmannian-sections}
Let $\lambda \vdash n$ be a partition. Define the support of $\lambda$ as
\[
\mathrm{supp}(\lambda) = \{\, i \mid d_i \neq 0 \,\}.
\]  
Let $X_{\mathrm{supp}}$ be the image of $G/B$ in the product 
\[
\prod_{i \in \mathrm{supp}(\lambda)} \Gr(i,n),
\]
which is a partial flag variety(Definition~\ref{def:partial-flag-variety}).

Let
$ \pi_{\mathrm{supp}} : G/B \longrightarrow X_{\mathrm{supp}}$
be the natural projection. Then the global sections of $L_\lambda$ can be computed on the smaller partial flag variety:
\[
\mathrm{H}^0(G/B, L_\lambda) \;\cong\; \mathrm{H}^0\Big(X_{\mathrm{supp}}, \bigotimes_{i \in \mathrm{supp}(\lambda)} \mathcal{O}_{\Gr(i,n)}(d_i) \Big|_{X_{\mathrm{supp}}} \Big).
\]
\end{proposition}

\begin{proof}
By the Picard group description of the flag variety, any line bundle on $G/B$ can be written uniquely as a tensor product of pullbacks of Plücker line bundles:
\[
L_\lambda \cong \bigotimes_{i=1}^{n-1} \pi_i^* \mathcal{O}_{\Gr(i,n)}(d_i).
\]
Let $X_{\mathrm{supp}}$ be the partial flag variety corresponding to the indices $i$ with $d_i \neq 0$, and let $\pi_{\mathrm{supp}}: G/B \to X_{\mathrm{supp}}$ be the natural projection. Then there exists a line bundle $M_\lambda$ on $X_{\mathrm{supp}}$ such that
\[
L_\lambda \cong \pi_{\mathrm{supp}}^* M_\lambda.
\]

Since $\pi_{\mathrm{supp}}$ has connected projective fibers, $(\pi_{\mathrm{supp}})_* \mathcal{O}_{G/B} \cong \mathcal{O}_{X_{\mathrm{supp}}}$;
using~\cite[Tag 0E0L]{stacks-project}. 
By the projection formula~\cite[Tag 01E8]{stacks-project},
\[
(\pi_{\mathrm{supp}})_* (\pi_{\mathrm{supp}}^* M_\lambda) \;\cong\; (\pi_{\mathrm{supp}})_* \mathcal{O}_{G/B} \otimes M_\lambda \;\cong\; M_\lambda.
\]
Taking global sections then gives the desired isomorphism:
\[
\mathrm{H}^0(G/B, L_\lambda) \;\cong\; \mathrm{H}^0(X_{\mathrm{supp}}, M_\lambda).
\]
\end{proof}

\begin{example}
For \( n = 4 \), the flag variety \( G/B \) can be described as
\[
X = G/B = \{ (p, l, \pi) \in \mathbb{P}^3 \times \mathrm{Gr}(2,4) \times \mathbb{P}^{3*} \mid p \in l \subseteq \pi \},
\]
where
\begin{itemize}
    \item \( p \in \mathbb{P}^3 \) is a point in projective 3-space,
    \item \( l \in \mathrm{Gr}(2,4) \) is a line (2-dimensional subspace) in \( \mathbb{K}^4 \),
    \item \( \pi \in \mathbb{P}^{3*} \) is a hyperplane in \( \mathbb{P}^3 \).
\end{itemize}

Line bundles on \( G/B \) are of the form
\[
L_{\lambda} = \pi_1^* \mathcal{O}(d_1) \otimes \pi_2^* \mathcal{O}(d_2) \otimes \pi_3^* \mathcal{O}(d_3),
\]
where the projections are
\[
\pi_1: G/B \to \mathbb{P}^3, \quad \pi_2: G/B \to \mathrm{Gr}(2,4), \quad \pi_3: G/B \to \mathbb{P}^{3*}.
\]

Let us consider the case \( d_1 = d_2 = 0 \), so that
\[
L_{\lambda} = \pi_3^* \mathcal{O}(d_3).
\]
Then the global sections satisfy
\[
\mathrm{H}^0(G/B, L_{\lambda}) = \mathrm{H}^0(G/B, \pi_3^* \mathcal{O}(d_3)) \cong \mathrm{H}^0\big(\mathbb{P}^{3*}, \pi_{3*} \pi_3^* \mathcal{O}(d_3)\big).
\]
By the projection formula
\[
\pi_{3*} \pi_3^* \mathcal{O}(d_3) \cong \pi_{3*} \mathcal{O}_{G/B} \otimes \mathcal{O}(d_3),
\]
so that
\[
\mathrm{H}^0(G/B, \pi_3^* \mathcal{O}(d_3)) \cong \mathrm{H}^0\big(\mathbb{P}^{3*}, \pi_{3*} \mathcal{O}_{G/B} \otimes \mathcal{O}(d_3)\big).
\]
Since the pushforward satisfies
\[
\pi_{3*} \mathcal{O}_{G/B} \cong \mathcal{O}_{\mathbb{P}^{3*}},
\]
it follows that
\[
\mathrm{H}^0(G/B, \pi_3^* \mathcal{O}(d_3)) \cong \mathrm{H}^0(\mathbb{P}^{3*}, \mathcal{O}(d_3)).
\] 
\end{example}

\begin{example}\label{example:flagvarietiesforpartition4}
The partitions of \(4\) are \((4), (3,1), (2,2), (2,1,1), (1,1,1,1)\). We now discuss each partition one by one.

\begin{enumerate}
  \item[(i)] {\textbf{\boldmath$\lambda = (4,0,0,0):$}} \\
  The highest weight is
  \[
  d_{(4,0,0,0)} = (4-0, 0-0, 0-0) = (4, 0, 0).
  \]
  The partial flag variety reduces to \(\mathbb{P}^3\), and the space of global sections is
  \[
  w_{(4,0,0,0)} \in \mathrm{H}^0\big(\mathbb{P}^3, \mathcal{O}(4)\big).
  \]

 \item[(ii)] {\textbf{\boldmath$\lambda = (3,1,0,0):$}} \\
  Here,
  \[
  d_{(3,1,0,0)} = (3-1, 1-0, 0-0) = (2,1,0).
  \]
  The partial flag variety embeds into
  \[
  F(1,2;4) \subset \mathbb{P}^3 \times \mathrm{Gr}(2,4).
  \]
  Let \(\pi_1 : \mathbb{P}^3 \times \mathrm{Gr}(2,4) \to \mathbb{P}^3\) and \(\pi_2 : \mathbb{P}^3 \times \mathrm{Gr}(2,4) \to \mathrm{Gr}(2,4)\) be the natural projections onto the first and second factors, respectively. Then the associated line bundle is
  \[
  L_{(3,1,0,0)} = \pi_1^* \mathcal{O}(2) \otimes \pi_2^* \mathcal{O}(1),
  \]
  and the global sections lie in
  \[
  w_{(3,1,0,0)} \in \mathrm{H}^0\big(F(1,2;4), \mathcal{O}(2,1)\big).
  \]

 \item[(iii)] {\textbf{\boldmath$\lambda = (2,2,0,0):$}} \\
  Then
  \[
  d_{(2,2,0,0)} = (2-2, 2-0, 0-0) = (0,2,0).
  \]
  The partial flag variety reduces to \(\mathrm{Gr}(2,4)\), and
  \[
  w_{(2,2,0,0)} \in \mathrm{H}^0\big(\mathrm{Gr}(2,4), \mathcal{O}(2)\big).
  \]

  \item[(iv)] {\textbf{\boldmath$\lambda = (2,1,1,0):$}} \\
  The highest weight is
  \[
  d_{(2,1,1,0)} = (2-1, 1-1, 1-0) = (1,0,1).
  \]
  The partial flag variety embeds into
  \[
  F(1,3;4) \subset \mathbb{P}^3 \times \mathbb{P}^{3*}.
  \]
  Let \(\pi_1 : \mathbb{P}^3 \times \mathbb{P}^{3*} \to \mathbb{P}^3\) and \(\pi_3 : \mathbb{P}^3 \times \mathbb{P}^{3*} \to \mathbb{P}^{3*}\) denote the natural projections onto the first and second factors, respectively. The associated line bundle is
  \[
  L_{(2,1,1,0)} = \pi_1^* \mathcal{O}(1) \otimes \pi_3^* \mathcal{O}(1).
  \]
  Thus, the space of global sections is
  \[
  w_{(2,1,1,0)} \in \mathrm{H}^0\big(F(1,3;4), \mathcal{O}(1,1)\big).
  \]

    \item[(v)] {\textbf{\boldmath$\lambda = (1,1,1,1):$}} \\
  The highest weight tuple is
  \[
  d_{(1,1,1,1)} = (1 - 1,\, 1 - 1,\, 1 - 1) = (0,0,0).
  \]
  The associated line bundle is trivial. Therefore,
  \[
  w_{(1,1,1,1)} \in \mathrm{H}^0(G/B, \mathcal{O}) \cong \mathbb{K}.
  \]
This component does not help distinguish between the algebras, so we do not include it in the tables.
\end{enumerate}
\end{example}

\section{COMPUTATION OF DISCRETE INVARIANTS}\label{sec:discrete-invariants}

In this section, we present some examples of the computation of discrete invariants using Magma. Some parts of the code are very long, so they are not written out in full. The complete code is available at~\cite{vishal2025}. 

Recall from Definition \ref{def:Xlambda} that the associated geometric locus is the base locus of this linear system:
\[
X_\lambda(w) := \mathrm{Base}\big(\widetilde{w}_\lambda(U_\lambda^*)\big) \subseteq G/B.
\]

This $X_\lambda(w)$ is a geometric invariant of the algebra. We then compute discrete invariants of \(X_\lambda(w)\) for all partitions \(\lambda \vdash 4\) to distinguish the algebras.

\begin{definition}
Let \( X_\lambda(w) \) be a closed subscheme defined by the vanishing of sections $\widetilde{w}_\lambda(U_\lambda^*)$ of a globally generated line bundle \( L_\lambda \) on the flag variety \( G/B \). There is a one-to-one correspondence
\[
  \left\{ \lambda \text{ a partition with at most } n \text{ parts} \right\}
  \quad \longleftrightarrow \quad
  \left\{ \text{effective line bundles } L_\lambda \text{ on } G/B \right\},
\]
which allows us to label each such subscheme by a partition \( \lambda \); see \cite[page 393]{fulton1991representation}.
\end{definition}



\subsection{Computation for Algebra H}
\begin{example}\label{discrete-inv-4-algH}
Consider \(w_\lambda\), the component of the superpotential for \(\lambda = (4)\). Using the Schur-Weyl duality~\ref{schurweylduality}, this component is obtained as
\[
w_{(4)} = e_{(4)} \cdot w \in \mathbb{S}_{(4)}(V) \otimes U_{(4)} \cong \mathrm{Sym}^4(V),
\]
where $e_{(4)}$ is the idempotent calculated using the character formula on page~\pageref{eqchformula}.

The partial flag variety \( G/B \) reduces to projective space $\mathbb{P}^3,$
and the line bundle \( L_{(4)} \) corresponds to \( \mathcal{O}(4) \), as in Example~\ref{example:flagvarietiesforpartition4}. Therefore, the component \( w_{(4)} \) defines a quartic hypersurface in \( \mathbb{P}^3 \), and the scheme
\[
X_{(4)}(w) \subseteq \mathbb{P}^3
\]
is precisely the vanishing locus of this quartic polynomial.

Next, we show the Magma code for Algebra H (as seen in Example~\ref{example:algebraH}). We construct the projective space over the rational function field $\mathbb{Q}(f,h)$. Throughout, $f$ and $h$ are taken to be generic. The code is as follows:

\begin{flushleft}
\texttt{QQ := RationalField();} \\
\texttt{K<f,h> := FunctionField(QQ, 2);} \\
\texttt{P<x1,x2,x3,x4> := ProjectiveSpace(K, 3);} \\
\texttt{X := Scheme(P, w(4));}
\end{flushleft}

We now compute the discrete invariants, listed below:
\begin{flushleft}
\texttt{// Primary Components of the scheme X } \\
\texttt{[[Dimension(p), Degree(p)] : p in PrimaryComponents(X)];}\\
\texttt{> [ [ 2, 2 ], [ 2, 2 ] ]}
\end{flushleft}

\begin{flushleft}
\texttt{// Primary Components of the reduced subscheme of X} \\
\texttt{[[Dimension(p), Degree(p)] : p in PrimaryComponents(ReducedSubscheme(X))];}\\
\texttt{> [ [ 2, 1 ], [ 2, 2 ] ]}
\end{flushleft}

\begin{flushleft}
\texttt{// Primary Components of the singular subscheme of X} \\
\texttt{[[Dimension(p), Degree(p)] : p in PrimaryComponents(SingularSubscheme(X))];}\\
\texttt{> [ [ 2, 1 ], [ 1, 4 ], [ 1, 1 ] ]}
\end{flushleft}

\begin{flushleft}
\texttt{// Primary Components of the reduced subscheme of the singular subscheme of X} \\
\texttt{[[Dimension(p), Degree(p)] : p in PrimaryComponents(ReducedSubscheme(SingularSubscheme(X)))];}\\
\texttt{> [ [ 2, 1 ], [ 1, 1 ] ]}
\end{flushleft}

\end{example}

When working with schemes over a field \( \mathbb{K} \), the decomposition into primary components over \( \mathbb{K} \) may not reflect the true geometric structure. In particular, these components may fail to capture irreducibility over the algebraic closure \( \overline{\mathbb{K}} \). To address this, we consider finite extensions of \( \mathbb{K} \), where the primary components become absolutely irreducible. 

\begin{definition}[\textbf{Base Change of a Scheme}]\label{def:base-change}
Let \( X \) be a projective scheme over a field \( \mathbb{K} \), and let \( L \) be an extension field of \( \mathbb{K} \). The base change of \( X \) to \( L \), denoted
\[
X_L = X \times_{\mathrm{Spec}\,\mathbb{K}} \mathrm{Spec}\,L,
\]
is the scheme obtained by extending scalars from \( \mathbb{K} \) to \( L \).
\end{definition}

\begin{lemma}\textnormal{\cite[Tag 035X]{stacks-project}}\label{lemma:XLreduced}
For any field extension \( L \) of \( \mathbb{K} \), the scheme \( X \) is reduced if and only if \( X_L \) is reduced.
\end{lemma}

\begin{definition}[\textbf{Absolutely Irreducible}, \textnormal{\cite[Tag 0365]{stacks-project}}]\label{def:absolutely-irred}
Let \( X \) be a scheme over a field \( \mathbb{K} \). We say that \( X \) is absolutely irreducible over \( \mathbb{K} \) if the scheme \( X_L \) is irreducible for every field extension \( L \) of \( \mathbb{K} \).
\end{definition}

\begin{lemma}\textnormal{\cite[Tag 054R]{stacks-project}}\label{lemma:galextension-primarycomp}
Let \( X \) be an equidimensional projective scheme over a field \( \mathbb{K} \). Then there exists a finite Galois extension \( L/\mathbb{K} \) such that
\[
X_L = C_1 \cup C_2 \cup \cdots \cup C_t,
\]
where each \( C_i \) is a primary component of \( X_L \). Let \( R_i \) denote the reduced subscheme associated to \( C_i \) for all \( i \). Then:
\begin{enumerate}
    \item Each \( R_i \) is absolutely irreducible,
    \item \( \dim C_i = \dim R_i = \dim X \),
    \item \( \deg C_i \) does not depend on \( i \),
    \item \( {\deg C_i}/{\deg R_i} \) does not depend on \( i \),
    \item \( \deg X = \deg X_L = t \cdot \deg C_i. \)
\end{enumerate}
\end{lemma}

\begin{lemma}\label{equidim-absolirred}
Let \( X \subseteq \mathbb{P}^n \) be an equidimensional closed subscheme of dimension \( d \), and let \( L \subseteq \mathbb{P}^n \) be a linear subspace of codimension \( c \). Suppose that:
\begin{enumerate}
  \item \( \deg(X \cap L) = \deg(X) \),
  \item \( \dim(X \cap L) = d - c \), and
  \item \( X \cap L \) is absolutely irreducible.
\end{enumerate}
Then \( X \) is absolutely irreducible.
\end{lemma}

\begin{proof}
Assume that \(X\) is reducible. Then its ideal can be written as
\[
I = K \cap J,
\]
where \(K\) and \(J\) are proper ideals corresponding to distinct irreducible components of \(X\), each of dimension \(d\), and \(I \neq K, I \neq J\).

Intersecting with the linear subspace \(L\) corresponds to adding its ideal \(Z\), so
\[
I + Z = (K \cap J) + Z.
\]
By the modular law for ideals,
\[
I + Z = (K + Z) \cap (J + Z).
\]
Since \(X \cap L\) is absolutely irreducible, \(I + Z\) is a prime ideal. Thus,
\[
I + Z = K + Z \quad \text{or} \quad I + Z = J + Z.
\]
Without loss of generality, suppose
\[
I + Z = K + Z.
\]
From the additivity of degree for components, we have
\[
\deg(X) = \deg(K) + \deg(J).
\]
Because intersecting with \(L\) preserves degree by assumption,
\[
\deg(X \cap L) = \deg(X) = \deg(K) + \deg(J).
\]
On the other hand, 
\[
\deg(K + Z) \leq \deg(K).
\]
But since \(I + Z = K + Z\), it follows that
\[
\deg(X \cap L) = \deg(I + Z) = \deg(K + Z) \leq \deg(K),
\]
which contradicts $\deg(K) + \deg(J) = \deg(X \cap L) \leq \deg(K),$
because \(\deg(J) > 0\).

Hence, the assumption that \(X\) is reducible leads to a contradiction. Therefore, \(X\) is absolutely irreducible.
\end{proof}

Therefore, using Definition~\ref{def:absolutely-irred}, Lemma~\ref{lemma:galextension-primarycomp} and Lemma~\ref{equidim-absolirred}, we can verify whether each component of \(X\) is absolutely irreducible. Below we give a brief explanation of how absolute irreducibility is checked in Magma.

We intersect a primary component \texttt{Component1} with a randomly chosen linear space to obtain a curve $C$. The command \texttt{IsAbsolutelyIrreducible(C)} returns \texttt{true} if \(C\) is absolutely irreducible, and \texttt{false} otherwise.

If the command returns \texttt{false}, we can determine a suitable field extension by computing \texttt{FieldOfGeometricIrreducibility(C)}, which gives the algebraic closure of the base field in the function field of \(C\). This process can be repeated for all primary components of $X$ to verify absolute irreducibility. The code is given below:

\begin{flushleft}
\texttt{pc := PrimaryComponents(X);} \\
\texttt{Component1 := pc[1]; // first component} \\
\texttt{C := Curve(Intersection(Component1, Scheme(P1, [49*x1 - 17*x2 + 6*x3])));} \\
\texttt{IsAbsolutelyIrreducible(C);} \\
\texttt{// true  $\rightarrow$ component is absolutely irreducible}
\end{flushleft}

\begin{flushleft}
\texttt{// false $\rightarrow$ we set }\\
\texttt{KK := FieldOfGeometricIrreducibility(C);} \\
\texttt{XKK := ChangeRing(X, KK);} \\
\end{flushleft}

\begin{example}~\label{procedure-check}
The Primary components computed in Example~\ref{discrete-inv-4-algH} are not all absolutely irreducible, so we recompute them over \texttt{XKK} and repeat until all components are absolutely irreducible.

\begin{flushleft}
\texttt{// Primary Components of the scheme XKK } \\
\texttt{[[Dimension(p), Degree(p)] : p in PrimaryComponents(XKK)];}\\
\texttt{> [ [ 2, 2 ], [ 2, 1 ], [ 2, 1 ] ]}
\end{flushleft}

\begin{flushleft}
\texttt{// Primary Components of the reduced subscheme of XKK} \\
\texttt{[[Dimension(p), Degree(p)] : p in PrimaryComponents(ReducedSubscheme(XKK))];}\\
\texttt{> [ [ 2, 1 ], [ 2, 1 ], [ 2, 1 ] ]}
\end{flushleft}

\begin{flushleft}
\texttt{// Primary Components of the singular subscheme of XKK} \\
\texttt{[[Dimension(p), Degree(p)] : p in PrimaryComponents(SingularSubscheme(XKK))];}\\
\texttt{> [ [ 2, 1 ], [ 1, 2 ], [ 1 , 2 ] [ 1, 1 ] ]}
\end{flushleft}

\begin{flushleft}
\texttt{// Primary Components of the reduced subscheme of the singular subscheme of XKK} \\
\texttt{[[Dimension(p), Degree(p)] : p in PrimaryComponents(ReducedSubscheme(SingularSubscheme(XKK)))];}\\
\texttt{> [ [ 2, 1 ], [ 1, 1 ] ]}
\end{flushleft}
\end{example}

\subsection{Representation-Theoretic Setup for \(\lambda = (2,2)\)}
\noindent

Consider \(w_\lambda\), the component of the superpotential for $\lambda = (2,2)$, defined by
\[
w_{(2,2)} := e_{(2,2)} \cdot w \in \mathbb{S}_{(2,2)} V \otimes U_{(2,2)} \subseteq V^{\otimes 4}.
\]
As a \(\GL(V)\)-representation, we have $ \mathbb{S}_{(2,2)} V \otimes U_{(2,2)} \;\simeq\; (\mathbb{S}_{(2,2)} V)^{\oplus 2}$, since  $U_{(2,2)}$ is 2-dimensional. 

The two copies of \(\mathbb{S}_{(2,2)} V\) can be realized inside \(V^{\otimes 4}\) in terms of symmetric and exterior powers:
\[
\mathbb{S}_{(2,2)} V \subseteq \mathrm{Sym}^2(\bigwedge^2 V) \subseteq V^{\otimes 4} \quad \text{and} \quad 
\mathbb{S}_{(2,2)} V \subseteq \mathrm{Sym}^2(\mathrm{Sym}^2 V) \subseteq V^{\otimes 4}.
\]

Let \( C_4 \leq S_4 \) denote the cyclic subgroup generated by the cycle $p = (1234)$. 
Upon restricting the irreducible \(S_4\)-representation \(U_{(2,2)}\) to \(C_4\), we obtain
\[
U_{(2,2)}\big|_{C_4} \simeq L_0 \oplus L_2,
\]
where \(L_0\) and \(L_2\) denote the one-dimensional representations of \(C_4\) on which the generator $p$ acts by \(1\) and \(-1\), respectively.
 
Therefore, we can write
\[
\mathbb{S}_{(2,2)}V \otimes U_{(2,2)}
\;\simeq\;
(\mathbb{S}_{(2,2)}V \otimes L_0)
\;\oplus\;
(\mathbb{S}_{(2,2)}V \otimes L_2).
\]

As $\mathrm{GL}(V)$-representations, the two copies of $\mathbb{S}_{(2,2)} V$ are isomorphic. In particular, we can write the exact sequence
\[
0 \longrightarrow \mathbb{S}_{(2,2)}V \;\longrightarrow\; \mathrm{Sym}^2(\mathrm{Sym}^2 V) \;\longrightarrow\; \mathrm{Sym}^4(V) \longrightarrow 0,
\]
where
\[
\mathbb{S}_{(2,2)} V = \ker\big(\mathrm{Sym}^2(\mathrm{Sym}^2 V) \to \mathrm{Sym}^4(V)\big).
\]

To make these copies of $\mathbb{S}_{(2,2)}V$ explicit, we construct bases for the subspaces
\[
\mathrm{Sym}^2(\mathrm{Sym}^2 V) \subseteq V^{\otimes 4} \quad \text{and} \quad \mathrm{Sym}^2(\bigwedge^2 V) \subseteq V^{\otimes 4}.
\]
Using commutators and skew-commutators, we generate spanning sets for each subspace and then reduce them to linearly independent sets to form bases. Each basis has dimension 20. 

For convenience, we set $W_0 := \mathrm{Sym}^2(\bigwedge^2 V)$ and $W_1 := \mathrm{Sym}^2(\mathrm{Sym}^2 V)$, and we now find highest weight vectors in these subspaces using the coproduct action. Recall that the coproduct action of a Lie algebra element $g \in \mathfrak{gl}(V)$ on $V^{\otimes 4}$ is defined by
\[
g \cdot (v_1 \otimes v_2 \otimes v_3 \otimes v_4) 
= \sum_{i=1}^{4} v_1 \otimes \cdots \otimes v_{i-1} \otimes g(v_i) \otimes v_{i+1} \otimes \cdots \otimes v_4,
\]
so that $g$ acts on each tensor factor individually while leaving the others unchanged.

We choose a set of generators for the Lie algebra $\mathfrak{gl}(V)$, consisting of the positive root vectors $e_{ij}$ with $i < j$.  
Intersecting the nullspaces of the raising operators $e_{ij}$ with the previously constructed subspaces $W_0$ and $W_1$ produces the highest weight vectors 
\[
h_{0} \in W_0 \quad \text{and} \quad h_{1} \in W_1, 
\]
which are given by 
$$h_0 := -x_3x_4x_3x_4 + x_3x_4^{2}x_3 + x_4x_3^{2}x_4 - x_4x_3x_4x_3,$$ and 
$$ h_1:= x_3^2 x_4^2
- \frac{1}{2} x_3 x_4 x_3 x_4
- \frac{1}{2} x_3 x_4^2 x_3
- \frac{1}{2} x_4 x_3^2 x_4
- \frac{1}{2} x_4 x_3 x_4 x_3
+ x_4^2 x_3^2.$$
Next, we decompose the highest weight vector $h_1$ into its $C_4$-isotypic components. These components correspond to the characters $L_0$ and $L_2$. The idempotents in $\mathbb{K}[C_4]$ associated with these characters are given by
\[
e_{L_0} = \frac{1}{4}(1 + p + p^2 + p^3), 
\qquad
e_{L_2} = \frac{1}{4}(1 - p + p^2 - p^3),
\]
which gives
\[
h_1^0 := e_{L_0} \cdot h_1
\qquad \text{and} \qquad 
h_1^1 := e_{L_2} \cdot h_1.
\]

The expressions for $h_{1}^{0}$ and $h_{1}^{1}$ are given by 
\[
\begin{aligned}
h_1^0 &:= \frac{1}{4} x_3^2 x_4^2
- \frac{1}{2} x_3 x_4 x_3 x_4
+ \frac{1}{4} x_3 x_4^2 x_3
+ \frac{1}{4} x_4 x_3^2 x_4
- \frac{1}{2} x_4 x_3 x_4 x_3
+ \frac{1}{4} x_4^2 x_3^2,\\
h_1^1 &:= \frac{3}{4} x_3^2 x_4^2
- \frac{3}{4} x_3 x_4^2 x_3
- \frac{3}{4} x_4 x_3^2 x_4
+ \frac{3}{4} x_4^2 x_3^2.
\end{aligned}
\]
Thus, the three highest weight vectors are $h_0, \; h_1^0, \; h_1^1.$ It is easy to interpret $h_0$ as a function on $\mathrm{Gr}(2,4)$, while $h_1$ is not, so we split it as $h_1 = h_1^0 + h_1^1$.

To generate the full subrepresentations, we repeatedly apply the coproduct action of Lie algebra elements to each highest weight vector to obtain the matrices $M_0$, $M_1$, and $M_2$. Each matrix has size $10000 \times 256$; their rank is 20, which equals $\dim \mathbb{S}_{(2,2)} V$. Here,
\[
\text{lieAlg} = \{ e_1, e_2, e_3, e_4 \} \;\cup\; \{ e_{ij} \mid 1 \le j < i \le 4 \} 
\]
denotes the set of Cartan elements and lowering operators. The matrices are then defined in Magma as follows:
\begin{flushleft} 
\texttt{M0 := [ h\textsubscript{1}\textsuperscript{0} * coproduct(X) * coproduct(Y) * coproduct(Z) * coproduct(W) | X, Y, Z, W in lieAlg ];}
\end{flushleft}

\begin{flushleft} 
\texttt{M1 := [ h\textsubscript{0} * coproduct(X) * coproduct(Y) * coproduct(Z) * coproduct(W) | X, Y, Z, W in lieAlg ];}
\end{flushleft}

\begin{flushleft} 
\texttt{M2 := [ h\textsubscript{1}\textsuperscript{1} * coproduct(X) * coproduct(Y) * coproduct(Z) * coproduct(W) | X, Y, Z, W in lieAlg ];}
\end{flushleft}

Next, to interpret these highest weight vectors as functions on $\mathrm{Gr}(2,4)$, we construct isomorphisms $$ h_1^0 \cdot \mathfrak{gl}(V) \simeq \mathbb{S}_{(2,2)} V \otimes L_0 \subseteq V^{\otimes 4}, \quad
 h_1^1 \cdot \mathfrak{gl}(V) \simeq \mathbb{S}_{(2,2)} V \otimes L_2 \subseteq V^{\otimes 4},$$
and $$h_{0} \cdot \mathfrak{gl}(V) \simeq W_0  \subseteq V^{\otimes 4}.$$  
Let \(\phi_1^0 : V^{\otimes 4} \longrightarrow V^{\otimes 4}\) denote the linear map. We are interested in its restriction 
\(
\phi_1^0| : h_1^0 \cdot \mathfrak{gl}(V)\;\xrightarrow{\sim}\; h_{0} \cdot \mathfrak{gl}(V).
\)

For a vector \(X \in \mathbb{K}^{10000}\) in the \(\mathfrak{gl}(V)\)-orbit, $M_0$ sends it to $\mathbb{S}_{(2,2)} V \otimes L_0 \subseteq V^{\otimes 4}$, while $M_1$ sends it to $W_0 \subseteq V^{\otimes 4}$. We then solve for \(\phi_1^0\) making the diagram commute:
\begin{equation*}
\phi_1^0 (M_0 X) = M_1 X \quad \text{for all } X.
\end{equation*}

\begin{figure}[h!]
\centering
\begin{tikzcd}[row sep=0.5cm, column sep=0.5cm]
& \mathbb{K}^{256} \arrow[dd, "\phi_1^0"] & \\
 \mathbb{K}^{10000} \arrow[ur, "M_0"] \arrow[dr, swap, "M_1"] & & \\
&  \mathbb{K}^{256} &
\end{tikzcd}
\end{figure}

Similarly, we can define the map \(\phi_1^1\) using the highest weight vector \(h_1^1\). We find the matrices representing \(\phi_1^0\) and \(\phi_1^1\) in Magma using the following commands:
\begin{flushleft}
\texttt{S01 := Transpose(Solution(Transpose(M0), Transpose(M1)));}
\\
\texttt{S21 := Transpose(Solution(Transpose(M2), Transpose(M1)));} 
\end{flushleft}

Each of the matrices $S01$ and $S21$ has size  $256 \times 256$. Applying the maps to the highest weight vectors gives $\phi_1^0(h_1^0) = h_0$ and  $\phi_1^1(h_1^1) = h_0.$

We project \(w_{(2,2)}\) onto \(L_0\) and \(L_2\), obtaining
$$w_{(2,2)} = w_{(2,2)}^{\mathrm{triv}} + w_{(2,2)}^{\mathrm{sign}},$$
where 
$w_{(2,2)}^{\mathrm{triv}} := e_{L_{0}} \cdot w_{(2,2)} \quad \text{and } \quad
w_{(2,2)}^{\mathrm{sign}} := e_{L_{2}} \cdot w_{(2,2)}.$

We then obtain the vectors $f_0$ and $f_2$ by applying the linear maps $\phi_1^0$ and $\phi_1^1$ to the corresponding components of the superpotential:
$$f_0 := \phi_1^0\big(w_{(2,2)}^{\mathrm{triv}}\big), \qquad
f_2 := \phi_1^1\big(w_{(2,2)}^{\mathrm{sign}}\big).$$
Let $S = \{S_i\}$ be an ordered basis of $W_0 \subseteq V^{\otimes 4}$, we then solve the following system of equations for the coefficients $\beta_0^i$ and $\beta_2^i$:
$$ f_0 = \sum_i \beta_0^i \, S_i, \qquad f_2 = \sum_i \beta_2^i \, S_i.$$
Next, we describe the geometric locus $X_{(2,2)}(w) \subseteq \mathrm{Gr}(2,4).$
Let $R = \mathbb{K}[b_1, \dots, b_6]$ be the polynomial ring in six variables. Via the Plücker embedding, the Grassmannian $\mathrm{Gr}(2,4)$ is realized as a projective subvariety $ \mathrm{Gr}(2,4) \subseteq \mathbb{P}(\bigwedge^2 V) \cong \mathbb{P}^5,$
with homogeneous coordinate ring generated by $b_1, \dots, b_6$ and subject to the Plücker relation
$$Q_{\mathrm{Gr}} := b_1 b_6 - b_2 b_5 + b_3 b_4.$$
Finally, to obtain the quadrics in the Plücker coordinates, we write the basis elements $S_i$ as quadratic forms in $b_1, \dots, b_6$ and form the linear combinations
\[
\beta_0 = \sum_i \beta_0^i \, S_i(b_1,\dots,b_6), \qquad
\beta_2 = \sum_i \beta_2^i \, S_i(b_1,\dots,b_6).
\]
A few basis elements are
\[
S = \{4 b_1^2, 4 b_1 b_2, 4 b_2^2, 4 b_1 b_3, 4 b_2 b_3, 4 b_3^2,  \dots \}.
\]
We then define the ideal
$I := \langle Q_{\mathrm{Gr}}, \beta_0, \beta_2 \rangle$ and the corresponding projective scheme is
$X_{(2,2)}(w) := \mathrm{Proj}(R/I) \subseteq \mathbb{P}^5.$ Therefore, we can compute the discrete invariants as in the Example~\ref{discrete-inv-4-algH}.

Similar computations are done for the partitions $\lambda = (2,1,1)$ and $\lambda = (3,1)$; the code is available in the GitHub repository~\cite{vishal2025}.

\section{NOTATIONS AND DISCRETE INVARIANTS}

Following the notation for partitions from Example~\ref{notations-lambda-}, we define the associated geometric locus \(X_\lambda(w)\) as
\[
\begin{aligned}
X_{(4)}(w) &\; :=  X_{400}, & 
X_{(3,1)}(w) &\; := X_{210}, & 
X_{(2,2)}(w) &\; := X_{020}, \\
X_{(2,1,1)}(w) &\; := X_{101}, & 
X_{(1,1,1,1)}(w) &\; := X_{000}.
\end{aligned}
\]

We also introduce the following schemes and their associated discrete invariants:
\begin{itemize}
  \item Let $\mathrm{PC}(X)$ denote the set of primary components of the scheme $X$. 
  The dimension-degree data of the primary components of \( X \) is given by
  \[
  \mathrm{dimdegPC}(X) := \{ [\dim(C), \deg(C)] \mid C \in \mathrm{PC}(X) \}.
  \]

  \item The \textbf{reduced subscheme} of $X$ is denoted by $\bar{X}$, 
  with associated data
  \[
  \mathrm{dimdegPC}(\bar{X}) := \{ [\dim(C), \deg(C)] \mid C \in \mathrm{PC}(\bar{X}) \}. \] 
  
  \item The \textbf{singular subscheme} of $X$ is denoted by 
   $X^{\mathrm{sing}}$, 
  with associated data
  $$
  \mathrm{dimdegPC}(X^{\mathrm{sing}}) := \{ [\dim(C), \deg(C)] \mid C \in \mathrm{PC}(X^{\mathrm{sing}}) \}. 
  $$

  \item The \textbf{reduced subscheme of the singular subscheme} of $X$ is denoted by 
  $\bar{X}^{\mathrm{sing}}$, with associated data
  $$
  \mathrm{dimdegPC}(\bar{X}^{\mathrm{sing}}) := \{ [\dim(C), \deg(C)] \mid C \in \mathrm{PC}(\bar{X}^{\mathrm{sing}}) \}. 
   $$
\end{itemize}

For each partition $\lambda \vdash 4$, we provide dimension-degree data for \( X \), \( X^{\mathrm{sing}} \), and \( \bar{X}^{\mathrm{sing}} \). Entries are pairs \([\mathrm{dim}, \mathrm{deg}]\), with multiplicities indicated where relevant. An empty bracket entry [ ] indicates that the singular subscheme is empty (i.e., its dimension is $-1$), and hence the corresponding surface is smooth. A table entry of $0$ indicates that the corresponding component $w_{\lambda}$ is zero.

We note that some algebras have already been shown to be isomorphic in the literature. In their paper \cite{MR2529094}, the authors classify 4-dimensional regular domains of the form \( (kQ[x_1, x_2])_P[y_1, y_2; \sigma] \) up to isomorphism~\cite[Proposition~4.4]{MR2529094}. Using the notion of \emph{\(\Sigma\)-\(M\)-duality}~\cite[Definition~4.3]{MR2529094}, based on twist equivalence~\cite[Definition~3.3]{MR2529094}, they identify dual pairs - (E, J), (F, I), (N, P), (T, U), and (W, Z) which are isomorphic via exchange of \( x_i \)'s and \( y_i \)'s, and self-dual algebras: B, C, M, O, R, and S.

We now present the discrete invariants for the 77 known families of Koszul AS-regular algebras in Sections~5.1-5.4.

\clearpage

\subsection{Discrete invariants for \(\boldsymbol{\lambda = (4)}\)}
\noindent


\begin{table}[ht!]
\centering
\renewcommand{\arraystretch}{1.5}  
\small 
\begin{tabularx}{\textwidth}{|c|>{\RaggedRight\arraybackslash}X|
                                >{\RaggedRight\arraybackslash}X|
                                >{\RaggedRight\arraybackslash}X|}
\hline
\textbf{Algebra} & \(\boldsymbol{\mathrm{dimdegPC}}(X)\) & \(\boldsymbol{\mathrm{dimdegPC}}(X^{\mathrm{sing}})\) & \(\boldsymbol{\mathrm{dimdegPC}}(\bar{X}^{\mathrm{sing}})\) \\
\hline
 A  & [2,1], [2,3] & [1,3], [1,1]$\times$2, [0,8], [0,2] & [1,1]$\times$3 \\
\hline
 B  & [2,4] & [1,1]$\times$2, [0,8]$\times$2 & [1,1]$\times$2 \\
\hline
 C & [2,4] & [1,1]$\times$2, [0,6]$\times$2, [0,2]$\times$2 & [1,1]$\times$2 \\
\hline
 D & [2,1], [2,3] & [1,1]$\times$4, [0,4]$\times$3 & [1,1]$\times$4 \\
\hline
 (E,J) & [2,4] & [1,1]$\times$2, [0,8]$\times$2 & [1,1]$\times$2 \\
\hline
 (F,I) & [2,4] & [1,1]$\times$2, [0,8]$\times$2 & [1,1]$\times$2 \\
\hline
 G & [2,1], [2,3] & [1,1]$\times$4, [0,4]$\times$3 & [1,1]$\times$4 \\
\hline
 H & [2,1]$\times$2, [2,2] & [1,1], [1,2]$\times$2,  [2,1]  & [2,1], [1,1] \\
\hline
 K & [2,1]$\times$4 & [1,1]$\times$6 & [1,1]$\times$6 \\
\hline
 L & [2,1]$\times$4 & [1,1]$\times$6 & [1,1]$\times$6 \\
\hline
 M & [2,1]$\times$4 & [1,1]$\times$6 & [1,1]$\times$6 \\
\hline
 (N,P) & [2,4] & [1,1]$\times$2, [0,8], [0,4]$\times$2 & [1,1]$\times$2 \\
\hline
 O & [2,4] & [1,1]$\times$2, [0,8], [0,4]$\times$2 & [1,1]$\times$2 \\
\hline
 Q & [2,1], [2,3] & [1,1]$\times$4, [0,2]$\times$2, [0,4]$\times$2 & [1,1]$\times$4 \\
\hline
 R & [2,4] & [1,1]$\times$2, [0,8]$\times$2 & [1,1]$\times$2 \\
\hline
 S & [2,4] & [1,1]$\times$2, [0,4]$\times$2, [0,3]$\times$2, [0,2]$\times$2 & [1,1]$\times$2 \\
\hline
 (T,U) & [2,4] & [1,1]$\times$2, [0,4]$\times$3, [0,3], [0,2] & [1,1]$\times$2 \\
\hline
 V & [2,1], [2,3] & [1,1]$\times$4, [0,2]$\times$2, [0,4]$\times$2 & [1,1]$\times$4\\
\hline
 (W,Z) & [2,4] & [1,1]$\times$2, [0,8], [0,4]$\times$2 & [1,1]$\times$2 \\
\hline
 X & [2,2], [2,1]$\times$2 & [2,1], [1,2]$\times$2, [1,1] & [2,1], [1,1] \\
\hline
 Y & [2,1], [2,3] & [1,1]$\times$4, [0,4]$\times$2, [0,2]$\times$2 & [1,1]$\times$4 \\
\hline
\end{tabularx} 
\caption{Discrete invariants for Koszul Artin-Schelter regular algebras of dimension four from \cite{MR2529094}.}\label{discreteinvlambda4:tab1}
\end{table}

\clearpage

\pagebreak

\begin{table}[ht!]
\centering
\renewcommand{\arraystretch}{1.5}  
\setlength{\tabcolsep}{2pt} 
\small 

\begin{tabularx}{\textwidth}{|c|>{\RaggedRight\arraybackslash}X|
                                >{\RaggedRight\arraybackslash}X|
                                >{\RaggedRight\arraybackslash}X|}
\hline
\textbf{Algebra} & \(\boldsymbol{\mathrm{dimdegPC}}(X)\) & \(\boldsymbol{\mathrm{dimdegPC}}(X^{\mathrm{sing}})\) & \(\boldsymbol{\mathrm{dimdegPC}}(\bar{X}^{\mathrm{sing}})\) \\
\hline
Vancliff & [2,1]$\times$2, [2,2] & [1,1]$\times$5, [0,4]$\times$2 & [1,1]$\times$5   \\
\hline
Vancliff twist &  [2,1]$\times$2, [2,2] &  [1,1]$\times$5, [0,4]$\times$2 & [1,1]$\times$5   \\
\hline 
$S_\infty$ & [2,4] & [ ] & [ ] \\
\hline
$S_\infty$ twist & [2,4] & [ ] & [ ] \\
\hline
$S_{d,i}$ $(i=1)$	 & [2,1]$\times$2, [2,2] &  [1,2]$\times$2, [1,1] & [1,2]$\times$2, [1,1]   \\
\hline
$S_{d,i}$ $(i=2)$	  & [2,1]$\times$2, [2,2] &  [1,2]$\times$2, [1,1] & [1,2]$\times$2, [1,1]   \\
\hline
$S_{d,i}$ $(i=3)$	 & [2,1]$\times$2, [2,2] &  [1,2]$\times$2, [1,1] & [1,2]$\times$2, [1,1]    \\
\hline
$S_{d,i}$ $(i=4)$	  & [2,1]$\times$2, [2,2] &  [1,2]$\times$2, [1,1] & [1,2]$\times$2, [1,1]    \\
\hline
$S_{d,i}$ $(i=5)$	  & [2,1]$\times$2, [2,2] &  [1,2]$\times$2, [1,1] & [1,2]$\times$2, [1,1]   \\
\hline
$S_{d,i}$ $(i=6)$	  &  [2,1]$\times$2, [2,2] &  [1,2]$\times$2, [1,1] & [1,2]$\times$2, [1,1]    \\
\hline
$S_{d,i}$ twist $(i=1)$	& [2,4] & [ ] & [ ] \\
   \hline 
$S_{d,i}$ twist $(i=3)$  &  [2,4] & [ ] & [ ] \\
   \hline 
$S_{d,i}$ twist  $(i=5)$  & [2,4] & [ ] & [ ] \\
 \hline 
$\mathbb{K}_Q[x_1, x_2, x_3, x_4]$	 & [2,1]$\times$4 & [1,1]$\times$6 &  [1,1]$\times$6   \\
\hline
Clifford & [2,4] & [ ] & [ ] \\
   \hline 
Central Extensions twist & [2,1], [2,3] & [1,3]  & [1,3]   \\
\hline
Kirkman R & [2,4] & [1,1]$\times$2, [0,8]$\times$2 & [1,1]$\times$2   \\
\hline
Kirkman S & [2,4] & [0,1]$\times$6 & [0,1]$\times$6    \\
\hline
Kirkman T & [2,4] & [0,1]$\times$2 &  [0,1]$\times$2  \\
\hline
Cassidy-Vancliff 2 & [2,4] & [ ] & [ ] \\
\hline
Cassidy-Vancliff 3 & [2,4]  & [0,8], [0,3]  & [0,1]$\times$2  \\
\hline
$R(3,a)$ & [2,2]$\times$2 & [2,1], [1,6] & [2,1] \\
\hline
$A_5$ & [2,1], [2,3] & [1,3], [0,8] & [1,3], [0,1] \\
\hline
\end{tabularx} 

\caption{Discrete invariants for Koszul Artin-Schelter regular algebras of dimension four from \cite{MR1429334, Stafford1994, Davies_2016, ExoticElliptic, Cassidy2009Generalizations, Pym2013PoissonSA, LECOUTRE_2017, Shelton2001OnKA}.}\label{discreteinvlambda4:tab2}
\end{table}
\pagebreak

\begin{table}[ht!]
\centering
\renewcommand{\arraystretch}{1.4}  
\setlength{\tabcolsep}{2pt} 
\small 

\begin{tabularx}{\textwidth}{|c|>{\RaggedRight\arraybackslash}X|
                                >{\RaggedRight\arraybackslash}X|
                                >{\RaggedRight\arraybackslash}X|}

\hline
\textbf{Algebra} & \(\boldsymbol{\mathrm{dimdegPC}}(X)\) & \(\boldsymbol{\mathrm{dimdegPC}}(X^{\mathrm{sing}})\) & \(\boldsymbol{\mathrm{dimdegPC}}(\bar{X}^{\mathrm{sing}})\) \\
\hline
\(\mathcal{F}_{\substack{(0,-1,-1,2)\\(0,0,0,0)}}\) & [2,1]$\times$2, [2,2] & [1,2], [1,1]$\times$3, [0,5] & [1,2], [1,1]$\times$3 \\
\hline
\(\mathcal{F}_{\substack{(-1,-1,1,1)\\(0,0,0,0)}}\) & [2,1]$\times$2, [2,2] & [1,1]$\times$5, [0,4]$\times$2 &[1,1]$\times$5 \\
\hline
\(\mathcal{F}_{\substack{(0,-1,-1,2)\\(-1,0,-1,2)\\(0,0,0,0)}}\) & [2,1], [2,3] & [1,1]$\times$3, [0,7], [0,5]$\times$2 &[1,1]$\times$3 \\
\hline
\(\mathcal{F}_{\substack{(0,-1,-1,2)\\(0,0,0,0)\\(-1,-1,2,0)}}\) & [2,1], [2,3] & [1,2], [1,1], [0,10], [0,6] &[1,2], [1,1] \\ 
\hline
OreExt Type \(A_1\)  &  [2,1], [2,3] & [1,3], [0,8] & [1,3], [0,1] \\
\hline
OreExt Type \(A_2\) & [2,1], [2,3] & [1,3], [0,8] & [1,3], [0,1] \\
\hline
OreExt Type \(A_3\)  & [2,1], [2,3] & [1,3], [0,8] & [1,3], [0,1] \\
\hline
OreExt Type \(B_1\)  & [2,1], [2,3] & [1,3], [0,8] & [1,3], [0,1] \\
\hline
OreExt Type \(E_1\) & [2,1], [2,3] & [1,3], [0,8] & [1,3], [0,1] \\
\hline
OreExt Type \(E_2\) & [2,1], [2,3] & [1,3], [0,8] & [1,3], [0,1] \\
\hline
OreExt Type $H.$I  & [2,1], [2,3] & [1,3], [0,8] & [1,3], [0,1] \\
\hline
OreExt Type $H.$II & [2,1], [2,3] & [1,3], [0,8] & [1,3], [0,1] \\
\hline
OreExt Type \(S_1^{'}\) & [2,1]$\times$2, [2,2] & [1,2], [1,1]$\times$3, [0,5] & [1,2], [1,1]$\times$3 \\
\hline
OreExt Type \(S_2\) & [2,1]$\times$2, [2,2] & [1,2]$\times$2, [1,1]$\times$2 & [1,2]$\times$2, [1,1]$\times$2 \\
\hline
\end{tabularx} 
\caption{Discrete invariants for Koszul Artin-Schelter regular algebras of dimension four from \cite{grimley2016hochschild, MR1429334}.}\label{discreteinvlambda4:tab3}
\end{table}


\begin{remark}
Note that \( w_{(4)} = 0 \) for the following algebras: \\
    Central Extensions, $L(1,1,2)$, $E(3)$, 4D Sklyanin, 4D Sklyanin twist, Shelton-Tingey, $\mathbb{K}[x_1,x_2,x_3,x_4]$, $S(2,3)$, $S_{d,i}$ twist $(i = 2, 4, 6)$, Cassidy-Vancliff 1, $L(1,1,2)^{\sigma}$, $S(2,3)^{\sigma}$, Caines Algebra, 
    \(\mathcal{F}_{\substack{(0,-1,-1,2)\\(0,0,0,0)\\(2,-1,-1,0)}}\), Lie Algebra 1, Lie Algebra 2, Lie Algebra 3, Lie Algebra 4, $\mathfrak{sl_2}$, CentExt Type $B$, CentExt Type $E$, CentExt Type $H$, CentExt Type $S_1$, CentExt Type $S_1^{'}$, CentExt Type $S_2$
\end{remark}

\begin{remark}
For all algebras in Tables~\ref{discreteinvlambda4:tab1}, \ref{discreteinvlambda4:tab2}, and \ref{discreteinvlambda4:tab3}, we have 
$X = \bar{X},$ except for H, X, and \(R(3,a)\).
\end{remark}

\clearpage

\pagebreak

We give some examples illustrating the structure and geometric meaning of the discrete invariants; here \(X = X_{400}\).







 

\begin{example}[Algebra M]\leavevmode
\begin{itemize}
   \item \(\mathrm{dimdegPC}(X) = [2,1]\times4\)
   \item \(\mathrm{dimdegPC}(X^{\mathrm{sing}}) = [1,1]\times6\)
   \item \(\mathrm{dimdegPC}(\bar{X}^{\mathrm{sing}}) = [1,1]\times6\)
\end{itemize}

\(X\) is reducible, the union of four planes in \(\mathbb{P}^3\), each plane irreducible. It is singular along six lines.
\end{example}


 







\begin{example}[Algebra \(S_\infty\)]\leavevmode
\begin{itemize}
    \item $\mathrm{dimdegPC}(X) = [2,4]$
    \item \(\mathrm{dimdegPC}(X^{\mathrm{sing}}) = [\,]\)
    \item \(\mathrm{dimdegPC}(\bar{X}^{\mathrm{sing}}) = [\,]\)
\end{itemize}
 
\(X\) is a smooth quartic surface in \(\mathbb{P}^3\). 
\end{example}




\clearpage 

\subsection{Discrete invariants for \(\boldsymbol{\lambda = (2,2)}\)}
\noindent

\begin{table}[ht!]
\centering
\renewcommand{\arraystretch}{1.3}  
\small 

\begin{tabularx}{\textwidth}{|c|>{\RaggedRight\arraybackslash}X|
                                >{\RaggedRight\arraybackslash}X|
                                >{\RaggedRight\arraybackslash}X|}                   
\hline
\textbf{Algebra} & \(\boldsymbol{\mathrm{dimdegPC}}(X)\) & \(\boldsymbol{\mathrm{dimdegPC}}(X^{\mathrm{sing}})\) & \(\boldsymbol{\mathrm{dimdegPC}}(\bar{X}^{\mathrm{sing}})\) \\
\hline
 A & [2,3]$\times$2, [2,1]$\times$2 & [1,3]$\times$3, [1,1], [0,4]$\times$2 & [1,3], [1,1]$\times$3 \\
\hline
 B & [2,4]$\times$2 & [1,2]$\times2$, [0,5]$\times$2 & [1,2]$\times2$ \\
\hline
 C & [2,8] & [0,7]$\times$2, [0,1]$\times$2 & [0,1]$\times$4 \\
\hline
 D & [2,6], [2,1]$\times$2 & [1,1]$\times$5, [0,3]$\times$2, [0,2]$\times$2 & [1,1]$\times$5 \\
\hline
 (E,J) & [2,2]$\times$4 & [1,2]$\times4$ & [1,2]$\times4$ \\
\hline
 (F,I) & [2,8] & [0,7]$\times$2, [0,1]$\times$2 & [0,1]$\times$4 \\
\hline
 G & [2,3]$\times$2, [2,1]$\times$2 & [1,3], [1,1]$\times$5, [0,3]$\times$2 & [1,3], [1,1]$\times$5 \\
\hline
 H & [2,4]$\times$2 & [2,3]$\times$2, [1,14] & [2,1]$\times$2 \\
\hline
 K & [2,2]$\times$2, [2,1]$\times$4 & [1,1]$\times$10 & [1,1]$\times$10 \\
\hline
 L & [2,2]$\times$2, [2,1]$\times$4 & [1,1]$\times$10 & [1,1]$\times$10 \\
\hline
 M & [3,4] & [0,1]$\times$6 & [0,1]$\times$6 \\
\hline
 (N,P) & [2,2]$\times$4 & [1,2]$\times4$ & [1,2]$\times4$ \\
\hline
 O & [2,8] & [0,7]$\times$2 & [0,1]$\times$2 \\
\hline
 Q & [2,2]$\times3$, [2,1]$\times$2 & [1,2]$\times$2, [1,1]$\times5$ & [1,2]$\times$2, [1,1]$\times5$ \\
\hline
 R & [2,4]$\times$2 & [1,4], [0,7]$\times$2, [0,1]$\times$2 & [1,4], [0,1]$\times$2 \\
\hline
 S & [2,8] & [2,4], [0,2]$\times$2 & [2,4] \\
\hline
 (T,U) & [2,4]$\times2$ & [1,1]$\times$4, [0,6]$\times$2, [0,4]$\times$2 & [1,1]$\times$4 \\
\hline
 V & [2,2]$\times3$, [2,1]$\times$2 & [1,2]$\times$2, [1,1]$\times5$  & [1,2]$\times$2, [1,1]$\times5$  \\
\hline
 (W,Z) & [3,4] & [0,1]$\times$2 & [0,1]$\times$2 \\
\hline
 X & [2,4]$\times$2 & [2,3]$\times$2, [1,14] & [2,1]$\times$2 \\
\hline
 Y & [3,4] & [0,1]$\times$4 & [0,1]$\times$4 \\
\hline
\end{tabularx}

\caption{Discrete invariants for Koszul Artin-Schelter regular algebras of dimension four from \cite{MR2529094}.}\label{discreteinvlambda22:tab1}
\end{table}

\clearpage

\pagebreak

\begin{table}[ht!]
\centering
\renewcommand{\arraystretch}{1.2}  
\small 

\begin{tabularx}{\textwidth}{|c|>{\RaggedRight\arraybackslash}X|
                                >{\RaggedRight\arraybackslash}X|
                                >{\RaggedRight\arraybackslash}X|}
                                
\hline
\textbf{Algebra} & \(\boldsymbol{\mathrm{dimdegPC}}(X)\) & \(\boldsymbol{\mathrm{dimdegPC}}(X^{\mathrm{sing}})\) & \(\boldsymbol{\mathrm{dimdegPC}}(\bar{X}^{\mathrm{sing}})\) \\
\hline
$L(1,1,2)$ & [3,2]$\times$2 & [2,2], [0,1]$\times$2 & [2,2], [0,1]$\times$2     \\
\hline
$E(3)$ & [3,4] & [1,3],[0,4] &  [1,1]    \\
\hline
4D Sklyanin	& [3,4] &  [0,1]$\times$6 &  [0,1]$\times$6        \\
\hline
Vancliff & [2,4], [2,1]$\times$4  & [1,1]$\times$8, [0,2]$\times$4 &  [1,1]$\times$8    \\
\hline
Vancliff twist & [2,4], [2,1]$\times$4  & [1,1]$\times$8, [0,2]$\times$4 &  [1,1]$\times$8    \\
\hline
Shelton-Tingey & [3,2]$\times$2 & [2,2] &  [2,2]     \\
\hline
$S_\infty$ & [3,4] & [0,1]$\times$2 &  [0,1]$\times$2    \\
\hline
$S_\infty$ twist & [3,2]$\times$2 & [2,2] & [2,2]     \\
\hline
$S_{d,i}$ $(i=1)$ & [3,4] & [0,1]$\times$6 & [0,1]$\times$6     \\
\hline
$S_{d,i}$ $(i=2)$ & [3,4] & [0,1]$\times$6 & [0,1]$\times$6       \\
\hline
$S_{d,i}$ $(i=3)$ & [3,4] & [0,1]$\times$6 & [0,1]$\times$6       \\
\hline
$S_{d,i}$ $(i=4)$ & [3,4] & [0,1]$\times$6 & [0,1]$\times$6       \\
\hline
$S_{d,i}$ $(i=5)$ & [3,4]  & [0,1]$\times$6 & [0,1]$\times$6       \\
\hline
$S_{d,i}$ $(i=6)$ & [3,4] & [0,1]$\times$6 & [0,1]$\times$6      \\
\hline
$S_{d,i}$ twist $(i=1)$ & [3,2]$\times$2 & [2,2] &  [2,2]     \\
\hline
$S_{d,i}$ twist $(i=3)$ & [3,2]$\times$2  & [2,2] &  [2,2]     \\
\hline
$S_{d,i}$ twist $(i=5)$ & [3,2]$\times$2 & [2,2] &  [2,2] \\
\hline
$\mathbb{K}_Q[x_1, x_2, x_3, x_4]$	 & [2,1]$\times$8 &  [1,1]$\times$12  & [1,1]$\times$12    \\
\hline
Kirkman R & [3,2]$\times$2 & [2,2] &  [2,2]    \\
\hline
Kirkman S & [3,4] & [1,1], [0,2]$\times2$ &   [1,1]    \\
\hline
Kirkman T & [3,4] & [0,1] &  [0,1]     \\
\hline
Cassidy-Vancliff 1 & [3,2]$\times$2 & [2,2] & [2,2]     \\
\hline
Cassidy-Vancliff 2 & [3,4] & [3,2], [0,4] & [3,2]     \\
\hline
Cassidy-Vancliff 3 & [3,2]$\times$2 & [2,1]$\times$2, [1,4] &  [2,1]$\times$2   \\
\hline
$R(3,a)$ & [2,2]$\times$2, [2,4] & [2,1]$\times$2, [1,12], [0,26] & [2,1]$\times$2 \\
\hline
$L(1,1,2)^{\sigma}$ & [3,2]$\times$2 & [2,2], [0,1]$\times$2 & [2,2], [0,1]$\times$2  \\
\hline
$A_5$ & [3,4] & [1,2], [0,3]$\times$2 & [1,1]\\
\hline
\end{tabularx} 

\caption{Discrete invariants for Koszul Artin-Schelter regular algebras of dimension four from \cite{MR1429334, Stafford1994, Davies_2016, ExoticElliptic, Cassidy2009Generalizations, Pym2013PoissonSA, LECOUTRE_2017, Shelton2001OnKA}.}\label{discreteinvlambda22:tab2}
\end{table}

\pagebreak

\begin{table}[ht!]
\centering
\renewcommand{\arraystretch}{1.4}  
\small 

\begin{tabularx}{\textwidth}{|c|>{\RaggedRight\arraybackslash}X|
                                >{\RaggedRight\arraybackslash}X|
                                >{\RaggedRight\arraybackslash}X|}
                                
\hline
\textbf{Algebra} & \(\boldsymbol{\mathrm{dimdegPC}}(X)\) & \(\boldsymbol{\mathrm{dimdegPC}}(X^{\mathrm{sing}})\) & \(\boldsymbol{\mathrm{dimdegPC}}(\bar{X}^{\mathrm{sing}})\) \\
\hline
\(\mathcal{F}_{\substack{(0,-1,-1,2)\\(0,0,0,0)}}\) & [2,1]$\times$8 & [1,1]$\times$12 & [1,1]$\times$12 \\
\hline
\(\mathcal{F}_{\substack{(-1,-1,1,1)\\(0,0,0,0)}}\) & [2,4], [2,1]$\times$4 & [1,1]$\times$8, [0,2]$\times$4 & [1,1]$\times$8\\
\hline
\(\mathcal{F}_{\substack{(0,-1,-1,2)\\(0,0,0,0)\\(2,-1,-1,0)}}\) & [3,2]$\times$2 & [2,2], [0,1]$\times$2 & [2,2], [0,1]$\times$2 \\
\hline
\(\mathcal{F}_{\substack{(0,-1,-1,2)\\(-1,0,-1,2)\\(0,0,0,0)}}\) & [2,1]$\times$8 & [1,1]$\times$12 & [1,1]$\times$12 \\
\hline
\(\mathcal{F}_{\substack{(0,-1,-1,2)\\(0,0,0,0)\\(-1,-1,2,0)}}\) & [2,1]$\times$8 & [1,1]$\times$12 & [1,1]$\times$12 \\
\hline
OreExt Type \(A_1\) & [3,4] & [0,1]$\times$6 & [0,1]$\times$6 \\
\hline
OreExt Type \(A_2\) & [3,4] & [0,1]$\times$6 & [0,1]$\times$6 \\
\hline
OreExt Type \(A_3\) & [3,4] & [0,1]$\times$6 & [0,1]$\times$6 \\
\hline
OreExt Type \(B_1\)  & [3,2]$\times$2 & [2,1]$\times$2, [1,4] & [2,1]$\times$2 \\
\hline
OreExt Type \(E_1\) & [2,6], [2,1]$\times$2 & [1,3]$\times$2 & [1,3]$\times$2 \\
\hline
OreExt Type \(E_2\) & [2,6], [2,1]$\times$2 & [1,3]$\times$2 & [1,3]$\times$2 \\
\hline
OreExt Type $H.$I & [2,2]$\times$3, [2,1]$\times$2 & [2,1]$\times$2, [1,5], [1,2]$\times$2, [1,1]$\times$2 & [2,1]$\times$2,[1,1]$\times$2 \\
\hline
OreExt Type $H.$II  & [2,2]$\times$3, [2,1]$\times$2 & [2,1]$\times$2, [1,5], [1,2]$\times$2, [1,1]$\times$2 & [2,1]$\times$2,[1,1]$\times$2 \\
\hline
OreExt Type \(S_1^{'}\)  & [2,1]$\times$8 & [1,1]$\times$12 & [1,1]$\times$12 \\
\hline
OreExt Type \(S_2\)  & [2,2]$\times$2, [2,1]$\times$4 & [1,1]$\times$10 & [1,1]$\times$10 \\
\hline
CentExt Type $E$  & [3,4] & [0,1]$\times$6 & [0,1]$\times$6 \\
\hline
CentExt Type $H$  & [3,2]$\times$2 & [2,1]$\times$2, [1,4] & [2,1]$\times$2 \\
\hline
CentExt Type \(S_1\) & [3,4] & [0,1]$\times$6 & [0,1]$\times$6 \\
\hline
CentExt Type \(S_1^{'}\)  & [3,2]$\times$2 & [2,2], [0,1]$\times$2 & [2,2], [0,1]$\times$2 \\
\hline
CentExt Type \(S_2\) & [3,4] & [0,1]$\times$6 & [0,1]$\times$6 \\
\hline
\end{tabularx} 
\caption{Discrete invariants for Koszul Artin-Schelter regular algebras of dimension four from \cite{grimley2016hochschild, MR1429334}.}\label{discreteinvlambda22:tab3}
\end{table}



\begin{remark}
Note that \(w_{(2,2)} = 0\) for the following algebras: \\
Central Extensions, 4D Sklyanin twist, $S_{d,i}$ twist $(i=2,4,6)$, $\mathbb{K}[x_1,x_2,x_3,x_4]$, Clifford,  $S(2,3)$, Central Extensions twist, $S(2,3)^{\sigma}$, Caines Algebra, Lie Algebra 1, Lie Algebra 2, Lie Algebra 3, Lie Algebra 4, $\mathfrak{sl_2}$ and CentExt Type $B$
\end{remark}

\begin{remark}
For all algebras in Tables~\ref{discreteinvlambda22:tab1}, \ref{discreteinvlambda22:tab2}, and \ref{discreteinvlambda22:tab3}, we have 
 $X = \bar{X},$ except for H, S, X, Cassidy-Vancliff 2, $R(3,a)$, OreExt Type $H.$I and OreExt Type $H.$II.
\end{remark}

\clearpage

\pagebreak

The following examples are for \(X = X_{020}\).




\begin{example}[Algebra C]\leavevmode
\begin{itemize}
    \item \(\mathrm{dimdegPC}(X) = [2,8]\)
    \item \(\mathrm{dimdegPC}(X^{\mathrm{sing}}) = [0,7]\times 2, [0,1]\times 2\)
    \item \(\mathrm{dimdegPC}(\bar{X}^{\mathrm{sing}}) = [0,1]\times 4\)
\end{itemize}

\(X\) is an irreducible surface of degree 8. It is singular at four points.
\end{example}

\begin{example}[Algebra M]\leavevmode
\begin{itemize}
    \item \(\mathrm{dimdegPC}(X) = [3,4]\)
    \item \(\mathrm{dimdegPC}(X^{\mathrm{sing}}) = [0,1]\times 6\)
    \item \(\mathrm{dimdegPC}(\bar{X}^{\mathrm{sing}}) = [0,1]\times 6\)
\end{itemize}

\(X\) is an irreducible 3-dimensional quartic surface, i.e., a quartic threefold. It is singular at six points.
\end{example}




\clearpage
\pagebreak

\subsection{Discrete invariants for \(\boldsymbol{\lambda = (2,1,1)}\)}
\noindent

\begin{table}[ht!]
\centering
\renewcommand{\arraystretch}{1.3}  
\small 

\begin{tabularx}{\textwidth}{|c|>{\RaggedRight\arraybackslash}X|
                                >{\RaggedRight\arraybackslash}X|
                                >{\RaggedRight\arraybackslash}X|}

\hline
\textbf{Algebra} & \(\boldsymbol{\mathrm{dimdegPC}}(X)\) & \(\boldsymbol{\mathrm{dimdegPC}}(X^{\mathrm{sing}})\) & \(\boldsymbol{\mathrm{dimdegPC}}(\bar{X}^{\mathrm{sing}})\) \\
\hline
A & [3,12], [3,4]$\times$2 &
[3,6], [2,6]$\times$2, [0,1]$\times$2 &
[3,6]  \\
\hline
B & [2,1]$\times$8, [2,2]$\times$6 & [1,1]$\times$24 &
[1,1]$\times$24 \\
\hline
C & [2,1]$\times$8, [2,2]$\times$6 & [1,1]$\times$24 &
[1,1]$\times$24 \\
\hline
D & [3,6]$\times$2, [3,4]$\times$2 & [2,3]$\times$4, [0,1]$\times$2 & [2,3]$\times$4, [0,1]$\times$2  \\
\hline
(E,J) & [2,16], [2,2]$\times$2 & [1,4]$\times$2 &
[1,4]$\times$2 \\
\hline
(F,I) & [3,6]$\times$2, [3,4]$\times$2 & [2,3]$\times$4, [0,1]$\times$2 & [2,3]$\times$4, [0,1]$\times$2  \\
\hline
G & [3,6]$\times$2, [3,4]$\times$2 & [2,3]$\times$4, [0,1]$\times$2 & [2,3]$\times$4, [0,1]$\times$2  \\
\hline
H & [3,6]$\times$2, [3,4]$\times$2 & [2,3]$\times$4 & [2,3]$\times$4 \\
\hline
K & [3,6]$\times$2, [3,4]$\times$2 & [2,3]$\times$4 & [2,3]$\times$4 \\
\hline
L & [3,6]$\times$2, [3,4]$\times$2 & [2,3]$\times$4 & [2,3]$\times$4 \\
\hline
M & [3,20] & [0,1]$\times$12 & [0,1]$\times$12 \\
\hline
O & [4,20] & [1,2]$\times$2 & [1,2]$\times$2 \\
\hline
Q & [2,3]$\times$4, [2,2]$\times$4 & [2,1]$\times$2, [1,3]$\times$4,[1,2]$\times$2, [1,1]$\times$4, [0,42] & [2,1]$\times$2, [1,1]$\times$4 \\
\hline
R & [2,2]$\times$6, [2,1]$\times$8 & [1,1]$\times$24 &
[1,1]$\times$24 \\
\hline
(T,U) & [4,10]$\times$2 & [3,6], [0,4] & [3,6] \\
\hline
V & [2,3]$\times$4, [2,2]$\times$4 & [2,1]$\times$2, [1,3]$\times$4,[1,2]$\times$2, [1,1]$\times$4, [0,42] & [2,1]$\times$2, [1,1]$\times$4 \\
\hline
(W,Z) & [4,20] & [1,2] & [1,2] \\
\hline
X & [3,6]$\times$2, [3,4]$\times$2 & [2,3]$\times$4 & [2,3]$\times$4 \\
\hline
Y & [4,20] & [1,2] & [1,2] \\
\hline
\end{tabularx}

\caption{Discrete invariants for Koszul Artin-Schelter regular algebras of dimension four from \cite{MR2529094}.}\label{discreteinvlambda211:tab1}
\end{table}

\clearpage

\pagebreak
\begin{table}[ht!]
\centering
\renewcommand{\arraystretch}{1.2}  
\setlength{\tabcolsep}{2pt} 
\small 

\begin{tabularx}{\textwidth}{|c|>{\RaggedRight\arraybackslash}X|
                                >{\RaggedRight\arraybackslash}X|
                                >{\RaggedRight\arraybackslash}X|}
                                
\hline
\textbf{Algebra} & \(\boldsymbol{\mathrm{dimdegPC}}(X)\) & \(\boldsymbol{\mathrm{dimdegPC}}(X^{\mathrm{sing}})\) & \(\boldsymbol{\mathrm{dimdegPC}}(\bar{X}^{\mathrm{sing}})\) \\
\hline
Vancliff & [2,1]$\times$8, [2,2]$\times$6 & [1,1]$\times$24 & [1,1]$\times$24 \\
\hline
Vancliff twist & [2,1]$\times$8, [2,2]$\times$6 & [1,1]$\times$24 & [1,1]$\times$24 \\
\hline
$S_\infty$ & [4,20] & [1,2]$\times$2 & [1,2]$\times$2 \\
\hline
$S_{d,i} (i=1,2,3,4,5,6)$ & [2,6]$\times$2, [2,1]$\times$4, [2,2]$\times$2 & [1,1]$\times$16 & [1,1]$\times$16  \\
\hline
$\mathbb{K}_Q[x_1, x_2, x_3, x_4]$ & [2,1]$\times$8, [2,2]$\times$6 & [1,1]$\times$24 & [1,1]$\times$24 \\
\hline
Central Extensions twist & [4,10]$\times$2 & [3,6] & [3,6] \\
\hline
Kirkman S & [4,10]$\times$2 & [3,6] & [3,6] \\
\hline
Kirkman T & [4,20] & [1,2] & [1,2] \\
\hline
$R(3,a)$ & [2,12], [2,4]$\times$2 &
[2,10], [2,3]$\times$2, [1,20]$\times$2 &
[2,1]$\times$2, [2,2]  \\
\hline
$L(1,1,2)^{\sigma}$ & [3,20] & [0,1]$\times$12 & [0,1]$\times$12  \\
\hline
$S(2,3)^{\sigma}$ & [4,10]$\times$2 & [3,6], [0,4] & [3,6]  \\
\hline 
$A_5$ & [3,6]$\times$2, [3,4]$\times$2 &
[2,3]$\times$2, [2,2]$\times$2, [2,1]$\times$2, [1,4]$\times$2 & [2,3]$\times$2, [2,2]$\times$2, [2,1]$\times$2 \\
\hline
\(\mathcal{F}_{\substack{(0,-1,-1,2)\\(0,0,0,0)}}\) & [2,1]$\times$8, [2,2]$\times$6 & [1,1]$\times$24 & [1,1]$\times$24 \\
\hline
\(\mathcal{F}_{\substack{(-1,-1,1,1)\\(0,0,0,0)}}\) & [2,1]$\times$8, [2,2]$\times$6 & [1,1]$\times$24 & [1,1]$\times$24 \\
\hline
\(\mathcal{F}_{\substack{(0,-1,-1,2)\\(0,0,0,0)\\(2,-1,-1,0)}}\) & [3,20] & [0,1]$\times$12 & [0,1]$\times$12\\
\hline
\(\mathcal{F}_{\substack{(0,-1,-1,2)\\(-1,0,-1,2)\\(0,0,0,0)}}\) & [2,1]$\times$8, [2,2]$\times$6 & [1,1]$\times$24 & [1,1]$\times$24 \\
\hline
\(\mathcal{F}_{\substack{(0,-1,-1,2)\\(0,0,0,0)\\(-1,-1,2,0)}}\) & [2,1]$\times$8, [2,2]$\times$6 & [1,1]$\times$24 & [1,1]$\times$24 \\
\hline 
\end{tabularx} 

\caption{Discrete invariants for Koszul Artin-Schelter regular algebras of dimension four from \cite{MR1429334, Stafford1994, Davies_2016, ExoticElliptic, Cassidy2009Generalizations, grimley2016hochschild, LECOUTRE_2017, Shelton2001OnKA}.}\label{discreteinvlambda211:tab2}
\end{table}

\begin{table}[ht!]
\centering
\renewcommand{\arraystretch}{1.4}  
\small 

\begin{tabularx}{\textwidth}{|c|>{\RaggedRight\arraybackslash}X|
                                >{\RaggedRight\arraybackslash}X|
                                >{\RaggedRight\arraybackslash}X|}

\hline
\textbf{Algebra} & \(\boldsymbol{\mathrm{dimdegPC}}(X)\) & \(\boldsymbol{\mathrm{dimdegPC}}(X^{\mathrm{sing}})\) & \(\boldsymbol{\mathrm{dimdegPC}}(\bar{X}^{\mathrm{sing}})\) \\
\hline
Lie Algebra 2 & [4,10]$\times$2 &  [3,6], [0,4] & [3,6] \\
\hline
Lie Algebra 3 & [4,10]$\times$2 &  [3,6], [0,4] & [3,6] \\
\hline
Lie Algebra 4 & [4,10]$\times$2 &  [3,6], [0,4] & [3,6] \\
\hline
OreExt Type \(A_1\) & [2,1]$\times$8, [2,2]$\times$6 & [1,1]$\times$24 & [1,1]$\times$24  \\
\hline
OreExt Type \(A_2\) & [2,1]$\times$8, [2,2]$\times$6 & [1,1]$\times$24 & [1,1]$\times$24 \\
\hline
OreExt Type \(A_3\)  & [2,1]$\times$8, [2,2]$\times$6 & [1,1]$\times$24 & [1,1]$\times$24 \\
\hline
OreExt Type \(B_1\) & [4,10]$\times$2 & [3,6], [0,4] & [3,6] \\
\hline
OreExt Type \(E_1\)  & [3,20] & [0,1]$\times$6 & [0,1]$\times$6 \\
\hline
OreExt Type \(E_2\)  & [3,20] & [0,1]$\times$6 & [0,1]$\times$6 \\
\hline
OreExt Type $H.$I  & [3,6]$\times$2, [3,4]$\times$2 &
[2,3]$\times$2, [2,2]$\times$2, [2,1]$\times$2, [1,4]$\times$2 & [2,3]$\times$2, [2,2]$\times$2, [2,1]$\times$2 \\
\hline
OreExt Type $H.$II & [3,6]$\times$2, [3,4]$\times$2 &
[2,3]$\times$2, [2,2]$\times$2, [2,1]$\times$2, [1,4]$\times$2 & [2,3]$\times$2, [2,2]$\times$2, [2,1]$\times$2 \\
\hline
OreExt Type \(S_1^{'}\)  & [2,1]$\times$8, [2,2]$\times$6 & [1,1]$\times$24 & [1,1]$\times$24 \\
\hline
OreExt Type \(S_2\) & [3,6]$\times$2, [3,4]$\times$2 &
[2,3]$\times$4 & [2,3]$\times$4 \\
\hline
CentExt Type $B$  & [4,10]$\times$2 & [3,6], [0,4] & [3,6] \\
\hline
CentExt Type $E$ & [3,20] & [0,1]$\times$6 & [0,1]$\times$6 \\
\hline
CentExt Type $H$ & [3,6]$\times$2, [3,4]$\times$2 &
[2,3]$\times$2, [2,2]$\times$2, [2,1]$\times$2, [1,4]$\times$2 & [2,3]$\times$2, [2,2]$\times$2, [2,1]$\times$2 \\
\hline
CentExt Type \(S_1\)  & [3,20] & [0,1]$\times$12 & [0,1]$\times$12 \\
\hline
CentExt Type \(S_1^{'}\)  & [3,20] & [0,1]$\times$12 & [0,1]$\times$12 \\
\hline
CentExt Type \(S_2\)  & [3,6]$\times$2, [3,4]$\times$2 &
[2,3]$\times$4 & [2,3]$\times$4 \\
\hline
\end{tabularx} 
\caption{Discrete invariants for Koszul Artin-Schelter regular algebras of dimension four from \cite{Jacobson, MR1429334}.}\label{discreteinvlambda211:tab3}
\end{table}



\FloatBarrier

\begin{remark}
Note that \(w_{(2,1,1)} = 0\) for the following algebras: \\
N, P, S, Central Extensions, $L(1,1,2)$, $E(3)$, 4D Sklyanin, 4D Sklyanin twist, Shelton-Tingey, $S_\infty$ twist, $S_{d,i}$ twist ($i=1,\ldots,6$), $\mathbb{K}[x_1,x_2,x_3,x_4]$,  Clifford, $S(2,3)$, Kirkman R, Cassidy-Vancliff 1, Cassidy-Vancliff 2, Cassidy-Vancliff 3, Caines Algebra, Lie Algebra 1, and \(\mathfrak{sl}_2\).
\end{remark}

\begin{remark}
For all algebras in Tables~\ref{discreteinvlambda211:tab1}, \ref{discreteinvlambda211:tab2}, and \ref{discreteinvlambda211:tab3}, we have $X = \bar{X},$ except for A, Q, V, and $R(3,a)$.
\end{remark}

\clearpage

\pagebreak

\subsection{Discrete invariants for \(\boldsymbol{\lambda = (3,1)}\)}
\noindent

Because the Magma code for the partition \((3,1)\), that is, those arising from \(X_{210}\), did not terminate, certain entries could not be computed. These missing entries are indicated by \textbf{\texttt{*****}}. As a result, invariants from this partition are used only in a few cases when distinguishing the algebras.

\begin{table}[ht!]
\centering
\renewcommand{\arraystretch}{1.2}  
\small 

\begin{tabularx}{\textwidth}{|c|>{\RaggedRight\arraybackslash}X|
                                >{\RaggedRight\arraybackslash}X|
                                >{\RaggedRight\arraybackslash}X|}
\hline
\textbf{Algebra} & \(\boldsymbol{\mathrm{dimdegPC}}(\bar{X})\) & \(\boldsymbol{\mathrm{dimdegPC}}(X^{\mathrm{sing}})\) & \(\boldsymbol{\mathrm{dimdegPC}}(\bar{X}^{\mathrm{sing}})\) \\
\hline
A  & [3,24], [3,11], [3,6], [3,5]$\times$3  & \textbf{\texttt{*****}} &   [3,11], [3,5]$\times$2, [2,5], [2,3]$\times$2   \\
\hline
 B & [3,5]$\times$2, [2,6]$\times$6, [2,4]$\times$4 & [2,1]$\times$4, [1,5]$\times$2, [1,1]$\times$8, [0,17]$\times$2, [0,18]$\times$2  &[2,1]$\times$4, [1,1]$\times$10   \\
\hline
 C  & [3,67], [3,5]$\times$2  & \textbf{\texttt{*****}}  & [2,8]$\times$2, [2,1]$\times$4    \\
\hline
 D  & [3,6], [3,5]$\times$4, [2,6]$\times$3, [2,4]$\times$2, [2,3]$\times$2 & [2,5], [2,9], [2,3]$\times$3, [2,1]$\times$5, [1,17]$\times$2, [1,16], [1,12], [1,11], [1,10], [1,7], [1,5], [0,44]$\times$2, [0,29]$\times$2, [0,72], [0,20]  &  [2,3]$\times$4, [2,1]$\times$6, [1,1]     \\
\hline
 (E,J) &  \textbf{\texttt{*****}} &  \textbf{\texttt{*****}}   & \textbf{\texttt{*****}}   \\
\hline
 (F,I) & [3,5]$\times$2, [2,10]$\times$2, [2,6]$\times$2, [2,4]$\times$4, [2,3]$\times$4  & [2,4]$\times$2, [2,1]$\times$4, [1,10], [1,7], [1,4]$\times$2, [1,3], [1,5], [1,1]$\times$6, [0,5]$\times2$, [0,21], [0,55], [0,47]  & [2,4]$\times$2, [2,1]$\times$4, [1,1]$\times$7    \\
\hline
 G  & [3,24], [3,11]$\times$2, [3,6], [3,5]$\times$4 & \textbf{\texttt{*****}}       &  [3,5], [2,7]$\times$2, [2,6], [2,5], [2,4]$\times$2, [2,3]$\times$3, [2,1]$\times$2    \\
\hline
H  & [3,6], [3,5]$\times2$, [2,7]$\times$2, [2,4]$\times$2 & \textbf{\texttt{*****}}  & [3,6], [2,3]$\times$2, [2,1]    \\
\hline
 K  & [3,6]$\times$2, [3,5]$\times$2, [2,14], [2,8] & [2,4]$\times$2, [2,3]$\times$2, [2,6], [1,16], [1,8]$\times$2, [1,12], [1,4]$\times$4, [0,20], [0,94], [0,168], [0,216]&  [2,6], [2,3]$\times$2, [2,1]$\times$2, [1,4]$\times$4    \\
\hline
 L  &  [3,6]$\times$2, [3,5]$\times$2, [2,14], [2,4]$\times$2 & [2,4]$\times$2, [2,3]$\times$4, [1,16], [1,8]$\times$2, [1,12], [1,4]$\times$2, [1,2]$\times$4, [0,49]$\times$2, [0,168], [0,276]  &[2,3]$\times$4, [2,1]$\times$2, [1,4]$\times$2, [1,2]$\times$4     \\
\hline
 M  & [3,67], [3,5]$\times$2 & \textbf{\texttt{*****}} & [2,3]$\times$4, [2,1]$\times$8    \\
\hline
 (N,P) &  [3,67], [3,5]$\times$2 & \textbf{\texttt{*****}} & [2,8], [2,3]$\times$2, [2,1]$\times$6    \\
\hline
 O  & [3,67], [3,5]$\times$2 & \textbf{\texttt{*****}} & [2,3]$\times$4, [2,1]$\times$8    \\
\hline
 Q  & \textbf{\texttt{*****}} & \textbf{\texttt{*****}} & [2,3], [2,1], [1,2]$\times$4, [1,1]$\times$11   \\
\hline
 R  & \textbf{\texttt{*****}} & \textbf{\texttt{*****}} & [1,1]$\times$18    \\
\hline
 S  & [3,35], [3,16]$\times$2, [3,5]$\times$2 &  \textbf{\texttt{*****}}  & [2,6]$\times$2, [2,4], [2,3]$\times$4, [2,1]$\times$6     \\
\hline
\end{tabularx} 
\caption{Discrete invariants for Koszul Artin-Schelter regular algebras of dimension four from \cite{MR2529094}.}
\end{table}

\pagebreak

\begin{table}[ht!]
\centering
\renewcommand{\arraystretch}{1.2}  
\setlength{\tabcolsep}{2pt} 
\small 

\begin{tabularx}{\textwidth}{|c|>{\RaggedRight\arraybackslash}X|
                                >{\RaggedRight\arraybackslash}X|
                                >{\RaggedRight\arraybackslash}X|}

\hline
\textbf{Algebra} & \(\boldsymbol{\mathrm{dimdegPC}}(\bar{X})\) & \(\boldsymbol{\mathrm{dimdegPC}}(X^{\mathrm{sing}})\) & \(\boldsymbol{\mathrm{dimdegPC}}(\bar{X}^{\mathrm{sing}})\) \\
\hline
 (T,U)  & \textbf{\texttt{*****}}  & \textbf{\texttt{*****}} &[2,3]$\times$2, [2,1]$\times$2, [1,1]$\times$9    \\
 \hline
 V  & \textbf{\texttt{*****}} & \textbf{\texttt{*****}}  & [2,3], [2,1], [1,2]$\times$4,  [1,1]$\times$11      \\
\hline
 (W,Z)   & [3,30], [3,16]$\times$2, [3,5]$\times$2  & \textbf{\texttt{*****}} & [3,5], [2,10]$\times$2, [2,8], [2,1]$\times$2   \\
\hline
 X  & [3,6], [3,5]$\times$2, [2,7]$\times$2, [2,4]$\times$2  & \textbf{\texttt{*****}} &  [3,6], [2,3]$\times$2, [2,1]   \\
\hline
 Y   & [3,24], [3,11]$\times$2, [3,6], [3,5]$\times$4  & \textbf{\texttt{*****}}  &  [3,5], [2,7]$\times$2, [2,6], [2,5], [2,4]$\times$2, [2,3]$\times$3, [2,1]$\times$2     \\
\hline
Central Extensions  & [4,56] &  [1,9]  & [1,9]  \\
\hline
$L(1,1,2)$  & [4,56] & [2,1]$\times$2, [1,4]$\times$2, [1,2]$\times$4, [1,3]$\times$2, [1,1]  & [2,1]$\times$2, [1,4]$\times$2, [1,1]    \\
\hline
$E(3)$  & [4,56] & [2,1], [1,14], [1,4], [1,1], [0,83]   &  [2,1], [1,4], [1,1]    \\
\hline
4D Sklyanin	 & [4,56]  &  [1,12]   &  [1,12]      \\
\hline
4D Sklyanin twist & [4,56] & [ ]   &  [ ]  \\
\hline
Vancliff  &  \textbf{\texttt{*****}} & \textbf{\texttt{*****}}  &  [2,3]$\times$6, [2,1]$\times$4    \\
\hline
Vancliff twist  & \textbf{\texttt{*****}} & \textbf{\texttt{*****}} & [2,3]$\times$6, [2,1]$\times$4      \\
\hline
Shelton-Tingey  & [4,56]  & [ ] &  [ ]      \\
\hline
$S_{d,i} (i=1,2,3,4,5,6)$ &  \textbf{\texttt{*****}} & \textbf{\texttt{*****}}    & [2,3]$\times$2, [2,1]$\times$2, [1,8]$\times$2, [1,1]$\times$2, [0,1]$\times$6    \\
\hline
$S_{d,i}$ twist $(i=2,4,6)$	  &  [4,56] & [ ] & [ ]      \\
\hline
$\mathbb{K}_Q[x_1, x_2, x_3, x_4]$  &  [3,6]$\times$4, [3,5]$\times$6  & \textbf{\texttt{*****}}  &[2,3]$\times$12, [2,1]$\times$4    \\
\hline
 $S(2,3)$  & [4,56] & \textbf{\texttt{*****}}  & [2,1], [1,1]$\times$7     \\
\hline
Cassidy-Vancliff 1  & [4,56] & [ ]  &  [ ]   \\
\hline
$R(3,a)$  & \textbf{\texttt{*****}} & \textbf{\texttt{*****}} & [2,3], [2,1] \\
\hline
$L(1,1,2)^{\sigma}$  & [3,6]$\times$2, [3,5], [2,3]$\times$6, [2,4], [2,1]$\times$2 &  \textbf{\texttt{*****}}  & [2,3]$\times$2, [2,1]$\times$2, [1,1]$\times$6, [0,1]$\times$6   \\
\hline
$S(2,3)^{\sigma}$  & [4,56] & \textbf{\texttt{*****}}  & [2,1], [1,1]$\times$7 \\
\hline
$A_5$  & \textbf{\texttt{*****}} & \textbf{\texttt{*****}} & [2,1], [1,14], [1,2]$\times$3, [1,1]$\times$6  \\
\hline
Caines Algebra  & [4,56] & [1,1], [0,7], [0,1]$\times$2    & [1,1], [0,1]$\times$2 \\
\hline
\(\mathcal{F}_{\substack{(0,-1,-1,2)\\(0,0,0,0)}}\) & [3,6]$\times$2, [3,5]$\times$3, [2,14], [2,4], [2,3]$\times$2  &  \textbf{\texttt{*****}}       & [2,3]$\times$6, [2,1]$\times$3, [1,4], [1,1]$\times$3 \\
\hline
\(\mathcal{F}_{\substack{(-1,-1,1,1)\\(0,0,0,0)}}\)  & [3,6]$\times$2, [3,5]$\times$5, [2,14], [2,4], [2,3]$\times$2 &  \textbf{\texttt{*****}}     &  [2,3]$\times$6, [2,1]$\times$4  \\
\hline
\(\mathcal{F}_{\substack{(0,-1,-1,2)\\(0,0,0,0)\\(2,-1,-1,0)}}\)  & [3,6]$\times$2, [3,5],[2,6]$\times$2, [2,4], [2,3]$\times$2, [2,2] & \textbf{\texttt{*****}}   & [2,3]$\times$2, [2,1]$\times$2, [1,1]$\times$6, [0,1]$\times$6 \\
\hline
\(\mathcal{F}_{\substack{(0,-1,-1,2)\\(-1,0,-1,2)\\(0,0,0,0)}}\)  & [3,6], [3,5]$\times$3, [2,16], [2,4]$\times$2, [2,3]$\times$4 & \textbf{\texttt{*****}}     & [2,3]$\times$6, [2,1]$\times$3, [1,1]$\times$2 \\
\hline
\end{tabularx} 
\caption{Discrete invariants for Koszul Artin-Schelter regular algebras of dimension four from \cite{MR2529094, MR1429334, Stafford1994, Davies_2016, ExoticElliptic, Cassidy2009Generalizations, grimley2016hochschild, LECOUTRE_2017, Shelton2001OnKA}.}
\end{table}

\begin{table}[ht!]
\centering
\renewcommand{\arraystretch}{1.2}
\small 

\begin{tabularx}{\textwidth}{|c|>{\RaggedRight\arraybackslash}X|
                                >{\RaggedRight\arraybackslash}X|
                                >{\RaggedRight\arraybackslash}X|}
\hline
\textbf{Algebra} & \(\boldsymbol{\mathrm{dimdegPC}}(\bar{X})\) & \(\boldsymbol{\mathrm{dimdegPC}}(X^{\mathrm{sing}})\) & \(\boldsymbol{\mathrm{dimdegPC}}(\bar{X}^{\mathrm{sing}})\) \\
\hline
\(\mathcal{F}_{\substack{(0,-1,-1,2)\\(0,0,0,0)\\(-1,-1,2,0)}}\)  & \textbf{\texttt{*****}} &  \textbf{\texttt{*****}} & [2,3]$\times$2, [2,1]$\times$2, [1,4], [1,1]$\times$5 \\ 
\hline
Lie Algebra 1  & [4,16]$\times$2, [4,24]  & [3,20], [3,6]$\times$2   & [3,5], [3,6]$\times$2 \\
\hline
Lie Algebra 2 & [4,16]$\times$2, [4,24] & [3,5]$\times$2, [3,6], [3,11], [3,12]  & [3,5]$\times$2, [3,6], [3,11] \\
\hline
Lie Algebra 3  & [4,16], [4,40] & [3,5]$\times$2, [3,6], [2,18]$\times$2, [2,5], [2,1] &  [3,6], [3,5]$\times$2, [1,1] \\
\hline
Lie Algebra 4  & [4,16], [4,40] &  [3,15], [3,6], [2,24], [1,7], [1,1]     &  [3,6], [3,5], [1,1] \\
\hline
\(\mathfrak{sl}_2\)  & [4,16], [4,40]& [3,6], [3,10], [2,24] & [3,6], [3,10] \\
\hline
OreExt Type \(A_1\)  & \textbf{\texttt{*****}} & [2,4], [1,21], [1,1]$\times$3, [0,1]$\times$18  & [2,1], [1,21], [1,1]$\times$3   \\
\hline
OreExt Type \(A_2\) & \textbf{\texttt{*****}} & [2,4], [1,21], [1,1]$\times$3, [0,1]$\times$18  & [2,1], [1,21], [1,1]$\times$3   \\
\hline
OreExt Type \(A_3\)  & \textbf{\texttt{*****}} & [2,4], [1,21], [1,1]$\times$3, [0,1]$\times$18  & [2,1], [1,21], [1,1]$\times$3   \\
\hline
OreExt Type \(B_1\)  & [3,35], [3,21], [3,15], [3,6] & \textbf{\texttt{*****}}   &  [2,12], [2,9]$\times$2, [2,4], [2,3]$\times$2, [2,1]$\times$4, [1,2] \\
\hline
OreExt Type \(E_1\)   & \textbf{\texttt{*****}} & \textbf{\texttt{*****}}  &  \textbf{\texttt{*****}} \\
\hline
OreExt Type \(E_2\) & \textbf{\texttt{*****}} &  \textbf{\texttt{*****}}  &  \textbf{\texttt{*****}} \\
\hline
OreExt Type $H.$I   & \textbf{\texttt{*****}}   & \textbf{\texttt{*****}}   & [2,1], [1,10], [1,2]$\times$3, [1,1]$\times$10 \\
\hline
OreExt Type $H.$II   & \textbf{\texttt{*****}}   & \textbf{\texttt{*****}}   & [2,1], [1,10], [1,2]$\times$3, [1,1]$\times$10 \\
\hline
OreExt Type \(S_1^{'}\)  & [3,6]$\times$2, [3,5]$\times$3, [2,14], [2,4], [2,3]$\times$2 & \textbf{\texttt{*****}}   &   [2,3]$\times$6, [2,1]$\times$3, [1,4], [1,1]$\times$3 \\
\hline
OreExt Type \(S_2\)  & [3,6]$\times$2, [3,5]$\times$2, [2,14], [2,4]$\times$2 &  \textbf{\texttt{*****}}   & [2,3]$\times$4, [2,1]$\times$2, [1,4]$\times$2, [1,2]$\times$4 \\
\hline
CentExt Type $B$   & [4,56] &  [1,5], [1,8], [1,2], [1,1]$\times$2   &   [1,5], [1,8], [1,2], [1,1]$\times$2\\
\hline
CentExt Type $E$  & [3,15]$\times$2, [3,6], [3,41] & \textbf{\texttt{*****}}  &   \textbf{\texttt{*****}}  \\
\hline
CentExt Type $H$   & [3,35], [3,15], [3,11], [3,10], [3,6] &  \textbf{\texttt{*****}}    &  [2,12], [2,9], [2,8], [2,6], [2,4], [2,3]$\times$3, [2,1]$\times$5 \\
\hline
CentExt Type \(S_1\)   & [3,23], [3,6]$\times$4, [3,5]$\times$6  & \textbf{\texttt{*****}} &[2,6]$\times$4, [2,3]$\times$12, [2,1]$\times$4 \\
\hline
CentExt Type \(S_1^{'}\)  & [3,35], [3,10], [3,6]$\times$2, [3,5]$\times$3 &  \textbf{\texttt{*****}}      & [3,5], [2,8], [2,6], [2,3]$\times$6, [2,1] \\
\hline
CentExt Type \(S_2\) & [3,35], [3,6]$\times$2, [3,5]$\times$6 & \textbf{\texttt{*****}}   & [2,6]$\times$2, [2,4]$\times$4, [2,3]$\times$8, [2,1]$\times$6 \\
\hline
\end{tabularx} 
\caption{Discrete invariants for Koszul Artin-Schelter regular algebras of dimension four from \cite{Jacobson, MR1429334}.}
\end{table}

\FloatBarrier
\vspace{0.5em}

\begin{remark}
Note that \(w_{(3,1)} = 0\) for the following algebras: \\
\(S_\infty\), \(S_\infty\) twist, $S_{d,i}$ twist ($i=1,3,5$), \(\mathbb{K}[x_1,x_2,x_3,x_4]\), Clifford, Central Extensions twist, Kirkman R, Kirkman S, Kirkman T, Cassidy-Vancliff 2, Cassidy-Vancliff 3.
\end{remark}

\section{Distinction of Components}

\begin{theorem}\label{maintheorem-inv}
We consider the algebraic stack $\mathcal{A}_4$ as defined in \cite[page~10]{bhatoy-colin-felix-ravali-components}. This stack has $45$ components (listed in \cite[Table~4]{bhatoy-colin-felix-ravali-components}).
Using computations performed in the computer algebra system Magma~\cite{magma}, we partitioned these components into 39 equivalence classes (referred to as boxes), where each class consists of algebras that share the same discrete invariants coming from the projective schemes \(X_{400}, X_{000}, X_{020}, X_{101}, X_{210}\).

If \(A\) and \(B\) are generic families of algebras lying in different equivalence classes, then \(A \not\simeq B\). It follows that components of \(\mathcal{A}_4\) lying in different equivalence classes (boxes) are distinct.
\end{theorem}

\begin{proof}
Each algebra \( A \) corresponds to a superpotential \( w_A \in V^{\otimes n} \), where \( \dim V = 4 \). Similarly, \( B \) corresponds to a superpotential \( w_B \in V^{\otimes n} \).

The group \( G = \mathrm{PGL}(V) \) acts naturally on \( V^{\otimes n} \) via diagonal automorphisms. An isomorphism between algebras \( A \) and \( B \) implies the existence of an element \( g \in G \) such that:
\[
w_B = g \cdot w_A.
\]
This action preserves the Schur-Weyl decomposition~\ref{schurweylduality}, which splits the tensor space into components labeled by partitions $\lambda:$
\[
w_A = \sum_{\lambda} w_{A,\lambda}, \quad w_B = \sum_{\lambda} g \cdot w_{A,\lambda}.
\]
Thus, the action of \( g \) transforms each component \( w_{A,\lambda} \) according to \( g \cdot w_{A,\lambda} \in \mathbb{S}_\lambda(V) \otimes U_\lambda \).

By Borel-Weil~\ref{borelweilthm}, we have $\mathbb{S}_\lambda(V) \cong \mathrm{H}^0(G/B, L_\lambda).$
Each component \( w_{A,\lambda} \) corresponds to a morphism:
\[
\tilde{w}_{A,\lambda}: U_\lambda^* \to \mathrm{H}^0(G/B, L_\lambda).
\]
Under the action of \( g \in G \), the image of this morphism transforms as:
\[
\mathrm{im}(\tilde{w}_{B,\lambda}) = g \cdot \mathrm{im}(\tilde{w}_{A,\lambda}) \subseteq \mathrm{H}^0(G/B, L_\lambda),
\]
which shows that the image subspace transforms equivariantly under the group action.

The action of \( G = \mathrm{PGL}(V) \) on the projective space \( \mathbb{P}(\mathrm{H}^0(G/B, L_\lambda)) \) is by projective automorphisms, which preserve the dimension and degree of the projective variety 
\[
X_\lambda(A) := \mathrm{Base}\big( \tilde{w}_{A,\lambda} \big) \subseteq G/B.
\]
Specifically, the dimension and degree of the primary components, as well as the number of primary components, are invariant under isomorphism.
\end{proof}

The following boxes summarize the equivalence classes of components described in Theorem~\ref{maintheorem-inv}.

\noindent\textbf{Box 1}

\renewcommand{\arraystretch}{1.2}
\resizebox{\textwidth}{!}{%
\begin{tabular}{|c|l|c|c|c|}
\hline
\textbf{Algebra} & \textbf{Schemes} & $\boldsymbol{\mathrm{dimdegPC}}(X)$ & $\boldsymbol{\mathrm{dimdegPC}}(X^{\mathrm{sing}})$ & $\boldsymbol{\mathrm{dimdegPC}}(\bar{X}^{\mathrm{sing}})$ \\
\hline
\multirow{3}{*}{B} 
  & \(X_{400}\)   & [2,4] & [1,1]$\times$2, [0,8]$\times$2 & [1,1]$\times$2 \\
  & \(X_{020}\)   & [2,4]$\times$2 & [1,2]$\times$2, [0,5]$\times$2 & [1,2]$\times$2 \\
  & \(X_{101}\)   & [2,1]$\times$8, [2,2]$\times$6 & [1,1]$\times$24 & [1,1]$\times$24 \\
\hline
\end{tabular}
}

\noindent\textbf{Box 2}

\resizebox{\textwidth}{!}{%
\begin{tabular}{|c|l|c|c|c|}
\hline
\textbf{Algebra} & \textbf{Schemes} & $\boldsymbol{\mathrm{dimdegPC}}(X)$ & $\boldsymbol{\mathrm{dimdegPC}}(X^{\mathrm{sing}})$ & $\boldsymbol{\mathrm{dimdegPC}}(\bar{X}^{\mathrm{sing}})$ \\
\hline
\multirow{3}{*}{C} 
 & \(X_{400}\)   & [2,4] & [1,1]$\times$2, [0,6]$\times$2, [0,2]$\times$2 & [1,1]$\times$2 \\
 & \(X_{020}\)   & [2,8] & [0,7]$\times$2, [0,1]$\times$2 & [0,1]$\times$4 \\
 & \(X_{101}\)  & [2,1]$\times$8, [2,2]$\times$6 & [1,1]$\times$24 & [1,1]$\times$24 \\
\hline
\end{tabular}
}

\noindent\textbf{Box 3}

\resizebox{\textwidth}{!}{%
\begin{tabular}{|c|l|c|c|c|}
\hline
\textbf{Algebra} & \textbf{Schemes} & $\boldsymbol{\mathrm{dimdegPC}}(X)$ & $\boldsymbol{\mathrm{dimdegPC}}(X^{\mathrm{sing}})$ & $\boldsymbol{\mathrm{dimdegPC}}(\bar{X}^{\mathrm{sing}})$ \\
\hline
\multirow{3}{*}{D} 
 & \(X_{400}\)   & [2,1], [2,3] & [1,1]$\times$4, [0,4]$\times$3 & [1,1]$\times$4 \\
 & \(X_{020}\)   & [2,6], [2,1]$\times$2 & [1,1]$\times$5, [0,3]$\times$2, [0,2]$\times$2 & [1,1]$\times$5 \\
 & \(X_{101}\)   & [3,6]$\times$2, [3,4]$\times$2  & [2,3]$\times$4, [0,1]$\times$2 &[2,3]$\times$4, [0,1]$\times$2 \\
\hline
\end{tabular}
}

\noindent\textbf{Box 4}

\renewcommand{\arraystretch}{1.2}
\resizebox{\textwidth}{!}{%
\begin{tabular}{|c|l|c|c|c|}
\hline
\textbf{Algebras} & \textbf{Schemes} & $\boldsymbol{\mathrm{dimdegPC}}(X)$ & $\boldsymbol{\mathrm{dimdegPC}}(X^{\mathrm{sing}})$ & $\boldsymbol{\mathrm{dimdegPC}}(\bar{X}^{\mathrm{sing}})$ \\
\hline
\multirow{3}{*}{(E, J)}    
& \(X_{400}\)   & [2,4] & [1,1]$\times$2, [0,8]$\times$2 & [1,1]$\times$2 \\ 
& \(X_{020}\)   & [2,2]$\times$4 & [1,2]$\times$4 & [1,2]$\times$4 \\
& \(X_{101}\)   & [2,16], [2,2]$\times$2 & [1,4]$\times$2 & [1,4]$\times$2 \\
\hline
\end{tabular}
}

\noindent\textbf{Box 5}

\resizebox{\textwidth}{!}{%
\begin{tabular}{|c|l|c|c|c|}
\hline
\textbf{Algebras} & \textbf{Schemes} & $\boldsymbol{\mathrm{dimdegPC}}(X)$ & $\boldsymbol{\mathrm{dimdegPC}}(X^{\mathrm{sing}})$ & $\boldsymbol{\mathrm{dimdegPC}}(\bar{X}^{\mathrm{sing}})$ \\
\hline
\multirow{3}{*}{(F, I)} 
  & \(X_{400}\)   & [2,4]  & [1,1]$\times$2, [0,8]$\times$2 & [1,1]$\times$2 \\
  & \(X_{020}\)   & [2,8]  & [0,7]$\times$2, [0,1]$\times$2 & [0,1]$\times$4 \\
  & \(X_{101}\)  & [3,6]$\times$2, [3,4]$\times$2  & [2,3]$\times$4, [0,1]$\times$2 &[2,3]$\times$4, [0,1]$\times$2 \\
\hline
\end{tabular}
}

\noindent\textbf{Box 6}

\resizebox{\textwidth}{!}{%
\begin{tabular}{|c|l|c|c|c|}
\hline
\textbf{Algebra} & \textbf{Schemes} & $\boldsymbol{\mathrm{dimdegPC}}(X)$ & $\boldsymbol{\mathrm{dimdegPC}}(X^{\mathrm{sing}})$ & $\boldsymbol{\mathrm{dimdegPC}}(\bar{X}^{\mathrm{sing}})$ \\
\hline
\multirow{3}{*}{G} 
 & \(X_{400}\)   & [2,1], [2,3] & [1,1]$\times$4, [0,4]$\times$3 & [1,1]$\times$4 \\
 & \(X_{020}\)   & [2,3]$\times$2, [2,1]$\times$2 & [1,3], [1,1]$\times$5, [0,3]$\times$2 & [1,3], [1,1]$\times$5 \\
 & \(X_{101}\)   & [3,6]$\times$2, [3,4]$\times$2  & [2,3]$\times$4, [0,1]$\times$2 & [2,3]$\times$4, [0,1]$\times$2 \\
\hline
\end{tabular}
}

\pagebreak

\noindent\textbf{Box 7}

\resizebox{\textwidth}{!}{%
\begin{tabular}{|c|l|c|c|c|}
\hline
\textbf{Algebra} & \textbf{Schemes} & $\boldsymbol{\mathrm{dimdegPC}}(X)$ & $\boldsymbol{\mathrm{dimdegPC}}(X^{\mathrm{sing}})$ & $\boldsymbol{\mathrm{dimdegPC}}(\bar{X}^{\mathrm{sing}})$ \\
\hline
\multirow{4}{*}{K} 
  & \(X_{400}\)  & [2,1]$\times$4 & [1,1]$\times$6 &  [1,1]$\times$6 \\
  & \(X_{020}\)  & [2,2]$\times$2, [2,1]$\times$4 & [1,1]$\times$10 & [1,1]$\times$10 \\
  & \(X_{101}\)  & [3,6]$\times$2, [3,4]$\times$2  & [2,3]$\times$4 &[2,3]$\times$4 \\
  & \(X_{210}\)
& \parbox[t]{5cm}{\centering [3,6]$\times$2, [3,5]$\times$2, [2,14], [2,8]\\[0.5ex] }
& \parbox[t]{6cm}{\centering [2,4]$\times$2, [2,3]$\times$2, [2,6], [1,16], [1,8]$\times$2, \\[0.5ex]
 [1,12], [1,4]$\times$4, [0,20], [0,94], [0,168], [0,276] \\[0.5ex] }
& \parbox[t]{5cm}{\centering [2,6], [2,3]$\times$2, [2,1]$\times$2, [1,4]$\times$4 \\[0.5ex] } \\
\hline
\end{tabular}
}

\noindent\textbf{Box 8}

\resizebox{\textwidth}{!}{%
\begin{tabular}{|c|l|c|c|c|}
\hline
\textbf{Algebra} & \textbf{Schemes} & $\boldsymbol{\mathrm{dimdegPC}}(X)$ & $\boldsymbol{\mathrm{dimdegPC}}(X^{\mathrm{sing}})$ & $\boldsymbol{\mathrm{dimdegPC}}(\bar{X}^{\mathrm{sing}})$ \\
\hline
\multirow{4}{*}{L} 
  & \(X_{400}\)  & [2,1]$\times$4 & [1,1]$\times$6 &  [1,1]$\times$6 \\
  & \(X_{020}\)  & [2,2]$\times$2, [2,1]$\times$4 & [1,1]$\times$10 & [1,1]$\times$10 \\
  & \(X_{101}\)  & [3,6]$\times$2, [3,4]$\times$2  & [2,3]$\times$4 &[2,3]$\times$4 \\
  & \(X_{210}\)
& \parbox[t]{5cm}{\centering [3,6]$\times$2, [3,5]$\times$2, [2,14], [2,4]$\times$2}
& \parbox[t]{6cm}{\centering [2,4]$\times$2, [2,3]$\times$4, [1,16], [1,8]$\times$2, [1,12],\\[0.5ex]
[1,4]$\times$2, [1,2]$\times$4, [0,10]$\times$2, [0,168], [0,276]}
& \parbox[t]{5cm}{\centering [2,3]$\times$4, [2,1]$\times$2, [1,4]$\times$2, [1,2]$\times$4} \\
\hline
\end{tabular}
}

\noindent\textbf{Box 9}

\resizebox{\textwidth}{!}{%
\begin{tabular}{|c|l|c|c|c|}
\hline
\textbf{Algebra} & \textbf{Schemes} & $\boldsymbol{\mathrm{dimdegPC}}(X)$ & $\boldsymbol{\mathrm{dimdegPC}}(X^{\mathrm{sing}})$ & $\boldsymbol{\mathrm{dimdegPC}}(\bar{X}^{\mathrm{sing}})$ \\
\hline
\multirow{3}{*}{M} 
  & \(X_{400}\)   & [2,1]$\times$4 & [1,1]$\times$6 & [1,1]$\times$6 \\
  & \(X_{020}\)   & [3,4] & [0,1]$\times$6 & [0,1]$\times$6 \\
  & \(X_{101}\) & [3,20] &  [0,1]$\times$12 & [0,1]$\times$12 \\
\hline
\end{tabular}
}

\noindent\textbf{Box 10}

\resizebox{\textwidth}{!}{%
\begin{tabular}{|c|l|c|c|c|}
\hline
\textbf{Algebras} & \textbf{Schemes} & $\boldsymbol{\mathrm{dimdegPC}}(X)$ & $\boldsymbol{\mathrm{dimdegPC}}(X^{\mathrm{sing}})$ & $\boldsymbol{\mathrm{dimdegPC}}(\bar{X}^{\mathrm{sing}})$ \\
\hline
\multirow{3}{*}{(N, P) } 
  & \(X_{400}\)  & [2,4] & [1,1]$\times$2, [0,8], [0,4]$\times$2 & [1,1] $\times$ 2 \\
  & \(X_{020}\)  & [2,2]$\times$4 & [1,2]$\times4$ & [1,2]$\times4$ \\
  & \(X_{101}\)  & 0 & 0 & 0 \\
\hline
\end{tabular}
}

\noindent\textbf{Box 11}

\resizebox{\textwidth}{!}{%
\begin{tabular}{|c|l|c|c|c|}
\hline
\textbf{Algebra} & \textbf{Schemes} & $\boldsymbol{\mathrm{dimdegPC}}(X)$ & $\boldsymbol{\mathrm{dimdegPC}}(X^{\mathrm{sing}})$ & $\boldsymbol{\mathrm{dimdegPC}}(\bar{X}^{\mathrm{sing}})$ \\
\hline
\multirow{3}{*}{O} 
  & \(X_{400}\)   & [2,4] & [1,1]$\times$2, [0,8], [0,4]$\times$2 & [1,1]$\times$2 \\
  & \(X_{020}\)   & [2,8] & [0,7]$\times$2 & [0,1]$\times$2 \\
  & \(X_{101}\)   & [4,20] & [1,2]$\times$2 & [1,2]$\times$2 \\
\hline
\end{tabular}
}

\noindent\textbf{Box 12}

\resizebox{\textwidth}{!}{%
\begin{tabular}{|c|l|c|c|c|}
\hline
\textbf{Algebras} & \textbf{Schemes} & $\boldsymbol{\mathrm{dimdegPC}}(X)$ & $\boldsymbol{\mathrm{dimdegPC}}(X^{\mathrm{sing}})$ & $\boldsymbol{\mathrm{dimdegPC}}(\bar{X}^{\mathrm{sing}})$ \\
\hline
\multirow{3}{*}{Q, V} 
  & \(X_{400}\)   & [2,1],[2,3]
 & [1,1]$\times$4, [0,2]$\times$2, [0,4]$\times$2 &[1,1]$\times$4   \\
  & \(X_{020}\)   & [2,2]$\times$3, [2,1]$\times$2 & [1,1]$\times$5, [1,2]$\times$2 & [1,1]$\times$5, [1,2]$\times$2   \\
  & \(X_{101}\)   & [2,3]$\times$4, [2,2]$\times$4 & [2,1]$\times$2, [1,3]$\times$4, [1,2]$\times$2, [1,1]$\times$4, [0,42] & [2,1]$\times$2, [1,1]$\times$4 \\
\hline
\end{tabular}
}

\begin{question}
Are the algebras Q and V isomorphic?
\end{question}

\pagebreak
\noindent\textbf{Box 13}

\resizebox{\textwidth}{!}{%
\begin{tabular}{|c|l|c|c|c|}
\hline
\textbf{Algebra} & \textbf{Schemes} & $\boldsymbol{\mathrm{dimdegPC}}(X)$ & $\boldsymbol{\mathrm{dimdegPC}}(X^{\mathrm{sing}})$ & $\boldsymbol{\mathrm{dimdegPC}}(\bar{X}^{\mathrm{sing}})$ \\
\hline
\multirow{3}{*}{R} 
  & \(X_{400}\)   & [2,4] & [1,1]$\times$2, [0,8]$\times$2 & [1,1]$\times$2 \\
  & \(X_{020}\)   & [2,4]$\times$2 & [1,4], [0,7]$\times$2, [0,1]$\times$2 & [1,4], [0,1]$\times$2 \\
  & \(X_{101}\)  & [2,1]$\times$8, [2,2]$\times$6 & [1,1]$\times$24 & [1,1]$\times$24 \\
\hline
\end{tabular}
}

\noindent\textbf{Box 14}

\resizebox{\textwidth}{!}{%
\begin{tabular}{|c|l|c|c|c|}
\hline
\textbf{Algebra} & \textbf{Schemes} & $\boldsymbol{\mathrm{dimdegPC}}(X)$ & $\boldsymbol{\mathrm{dimdegPC}}(X^{\mathrm{sing}})$ & $\boldsymbol{\mathrm{dimdegPC}}(\bar{X}^{\mathrm{sing}})$ \\
\hline
\multirow{3}{*}{S} 
 & \(X_{400}\)   & [2,4] & [1,1]$\times$2, [0,4]$\times$2, [0,3]$\times$2, [0,2]$\times$2 & [1,1]$\times$2 \\
 & \(X_{020}\)   & [2,8] & [2,4], [0,2]$\times$2 & [2,4] \\
 & \(X_{101}\)   & 0 & 0 & 0 \\
\hline
\end{tabular}
}

\noindent\textbf{Box 15}

\resizebox{\textwidth}{!}{%
\begin{tabular}{|c|l|c|c|c|}
\hline
\textbf{Algebras} & \textbf{Schemes} & $\boldsymbol{\mathrm{dimdegPC}}(X)$ & $\boldsymbol{\mathrm{dimdegPC}}(X^{\mathrm{sing}})$ & $\boldsymbol{\mathrm{dimdegPC}}(\bar{X}^{\mathrm{sing}})$ \\
\hline
\multirow{3}{*}{(T, U)} 
  & \(X_{400}\)   & [2,4]  & [1,1]$\times$2, [0,4]$\times$3, [0,3], [0,2]  & [1,1]$\times$2  \\
  & \(X_{020}\)   & [2,4]$\times$2  & [1,1]$\times$4,[0,6]$\times$2, [0,4]$\times$2 & [1,1]$\times$4 \\
  & \(X_{101}\)   & [4,10]$\times$2 & [3,6], [0,4]  & [3,6] \\
\hline
\end{tabular}
}

\noindent\textbf{Box 16}

\resizebox{\textwidth}{!}{%
\begin{tabular}{|c|l|c|c|c|}
\hline
\textbf{Algebras} & \textbf{Schemes} & $\boldsymbol{\mathrm{dimdegPC}}(X)$ & $\boldsymbol{\mathrm{dimdegPC}}(X^{\mathrm{sing}})$ & $\boldsymbol{\mathrm{dimdegPC}}(\bar{X}^{\mathrm{sing}})$ \\
\hline
\multirow{3}{*}{(W, Z)} 
  & \(X_{400}\)   & [2,4] & [1,1]$\times$2, [0,8] [0,4]$\times$2 & [1,1]$\times$2 \\
  & \(X_{020}\)   & [3,4] & [0,1]$\times$2 &  [0,1]$\times$2 \\
  & \(X_{101}\)   & [4,20] & [1,2] & [1,2] \\
\hline
\end{tabular}
}

\noindent\textbf{Box 17}

\renewcommand{\arraystretch}{1.2}
\resizebox{\textwidth}{!}{%
\begin{tabular}{|c|l|c|c|c|}
\hline
\textbf{Algebra} & \textbf{Schemes} & $\boldsymbol{\mathrm{dimdegPC}}(X)$ & $\boldsymbol{\mathrm{dimdegPC}}(X^{\mathrm{sing}})$ & $\boldsymbol{\mathrm{dimdegPC}}(\bar{X}^{\mathrm{sing}})$ \\
\hline
\multirow{3}{*}{Central Extensions} 
  & \(X_{400}\)   & 0 & 0 & 0 \\
  & \(X_{020}\)   & 0 & 0 & 0 \\
  & \(X_{101}\)   & 0 & 0 & 0 \\
\hline
\end{tabular}
}

\noindent\textbf{Box 18}

\renewcommand{\arraystretch}{1.2}
\resizebox{\textwidth}{!}{%
\begin{tabular}{|c|l|c|c|c|}
\hline
\textbf{Algebra} & \textbf{Schemes} & $\boldsymbol{\mathrm{dimdegPC}}(X)$ & $\boldsymbol{\mathrm{dimdegPC}}(X^{\mathrm{sing}})$ & $\boldsymbol{\mathrm{dimdegPC}}(\bar{X}^{\mathrm{sing}})$ \\
\hline
\multirow{3}{*}{Central Extensions twist} 
  & \(X_{400}\)   & [2,1], [2,3] & [1,3]  & [1,3] \\
  & \(X_{020}\)   & 0 & 0 & 0 \\
  & \(X_{101}\)   & [4,10]$\times$2 & [3,6] & [3,6] \\
\hline
\end{tabular}
}

\noindent\textbf{Box 19}

\resizebox{\textwidth}{!}{%
\begin{tabular}{|c|l|c|c|c|}
\hline
\textbf{Algebra} & \textbf{Schemes} & $\boldsymbol{\mathrm{dimdegPC}}(X)$ & $\boldsymbol{\mathrm{dimdegPC}}(X^{\mathrm{sing}})$ & $\boldsymbol{\mathrm{dimdegPC}}(\bar{X}^{\mathrm{sing}})$ \\
\hline
\multirow{3}{*}{4D Sklyanin} 
  & \(X_{400}\)   & 0 & 0 & 0 \\
  & \(X_{020}\)   & [3,4] & [0,1]$\times$6 & [0,1] $\times$ 6  \\
  & \(X_{101}\)   & 0 & 0 & 0 \\
\hline
\end{tabular}
}

\pagebreak
\noindent\textbf{Box 20}

\resizebox{\textwidth}{!}{%
\begin{tabular}{|c|l|c|c|c|}
\hline
\textbf{Algebra} & \textbf{Schemes} & $\boldsymbol{\mathrm{dimdegPC}}(X)$ & $\boldsymbol{\mathrm{dimdegPC}}(X^{\mathrm{sing}})$ & $\boldsymbol{\mathrm{dimdegPC}}(\bar{X}^{\mathrm{sing}})$ \\
\hline
\multirow{3}{*}{$S_{\infty}$} 
  & \(X_{400}\)  & [2,4] & [ ] &  [ ] \\
  & \(X_{020}\)   & [3,4] & [0,1]$\times$2 & [0,1]$\times$2 \\
  & \(X_{101}\)   &   [4,20]  & [1,2]$\times$2 &  [1,2]$\times$2 \\
\hline
\end{tabular}
}

\noindent\textbf{Box 21}

\resizebox{\textwidth}{!}{%
\begin{tabular}{|c|l|c|c|c|}
\hline
\textbf{Algebras} & \textbf{Schemes} & $\boldsymbol{\mathrm{dimdegPC}}(X)$ & $\boldsymbol{\mathrm{dimdegPC}}(X^{\mathrm{sing}})$ & $\boldsymbol{\mathrm{dimdegPC}}(\bar{X}^{\mathrm{sing}})$ \\
\hline
\multirow{3}{*}{$S_{d,i}$ $(i= 1,\ldots,6)$} 
  & \(X_{400}\)   & [2,1]$\times$2, [2,2] & [1,2]$\times$2, [1,1] & [1,2]$\times$2, [1,1]   \\
  & \(X_{020}\)   & [3,4]  & [0,1]$\times$6 & [0,1]$\times$6  \\
  & \(X_{101}\)   & [2,6]$\times$2, [2,1]$\times$4, [2,2]$\times$2  & [1,1]$\times$16 & [1,1]$\times$16    \\
\hline
\end{tabular}
}

\begin{remark} (see the remarks after \cite[Thm.~0.4]{{Stafford1994}})
Cyclically permuting the generators $x_1, x_2, x_3$ of $S_{d,i}$ defines an isomorphism with an algebra in the family $S_{d,i+2}$ (indices considered modulo 6): 
\[
S_{d,1} \simeq S_{d,3} \simeq S_{d,5} \quad \text{and} \quad 
S_{d,2} \simeq S_{d,4} \simeq S_{d,6}.
\]
\end{remark}

\noindent\textbf{Box 22}

\resizebox{\textwidth}{!}{%
\begin{tabular}{|c|l|c|c|c|}
\hline
\textbf{Algebra} & \textbf{Schemes} & $\boldsymbol{\mathrm{dimdegPC}}(X)$ & $\boldsymbol{\mathrm{dimdegPC}}(X^{\mathrm{sing}})$ & $\boldsymbol{\mathrm{dimdegPC}}(\bar{X}^{\mathrm{sing}})$ \\
\hline
\multirow{3}{*}{Clifford} 
  & \(X_{400}\)   & [2,4] & [ ] &  [ ]  \\
  & \(X_{020}\)   & 0 & 0 &  0 \\
  & \(X_{101}\)   & 0 & 0 &  0 \\
\hline
\end{tabular}
}

\noindent\textbf{Box 23}

\resizebox{\textwidth}{!}{%
\begin{tabular}{|c|l|c|c|c|}
\hline
\textbf{Algebra} & \textbf{Schemes} & $\boldsymbol{\mathrm{dimdegPC}}(X)$ & $\boldsymbol{\mathrm{dimdegPC}}(X^{\mathrm{sing}})$ & $\boldsymbol{\mathrm{dimdegPC}}(\bar{X}^{\mathrm{sing}})$ \\
\hline
\multirow{3}{*}{$\mathbb{K}_Q[x_1, x_2, x_3, x_4]$} 
  & \(X_{400}\)  & [2,1]$\times$4 & [1,1]$\times$6 &  [1,1]$\times$6  \\
  & \(X_{020}\)   & [2,1]$\times$8 &  [1,1]$\times$12  & [1,1]$\times$12    \\
  & \(X_{101}\)   & [2,1]$\times$8, [2,2]$\times$6 & [1,1]$\times$24 & [1,1]$\times$24 \\
\hline
\end{tabular}
}

\noindent\textbf{Box 24}

\resizebox{\textwidth}{!}{%
\begin{tabular}{|c|l|c|c|c|}
\hline
\textbf{Algebra} & \textbf{Schemes} & $\boldsymbol{\mathrm{dimdegPC}}(X)$ & $\boldsymbol{\mathrm{dimdegPC}}(X^{\mathrm{sing}})$ & $\boldsymbol{\mathrm{dimdegPC}}(\bar{X}^{\mathrm{sing}})$ \\
\hline
\multirow{3}{*}{Cassidy-Vancliff 1} 
  & \(X_{400}\)  & 0 & 0 &  0 \\
  & \(X_{020}\)   & [3,2]$\times$2 & [2,2] & [2,2]    \\
  & \(X_{101}\)   & 0 & 0 &  0 \\
\hline
\end{tabular}
}

\noindent\textbf{Box 25}

\resizebox{\textwidth}{!}{%
\begin{tabular}{|c|l|c|c|c|}
\hline
\textbf{Algebra} & \textbf{Schemes} & $\boldsymbol{\mathrm{dimdegPC}}(X)$ & $\boldsymbol{\mathrm{dimdegPC}}(X^{\mathrm{sing}})$ & $\boldsymbol{\mathrm{dimdegPC}}(\bar{X}^{\mathrm{sing}})$ \\
\hline
\multirow{3}{*}{ $R(3,a)$ } 
  & \(X_{400}\)  &  [2,2]$\times$2 & [2,1], [1,6] & [2,1] \\  
  & \(X_{020}\)   & [2,2]$\times$2, [2,4] & [2,1]$\times$2, [1,12], [0,26] & [2,1]$\times$2 \\
  & \(X_{101}\)   &[2,12], [2,4]$\times$2 & [2,10], [2,3]$\times$2, [1,20]$\times$2 & [2,1]$\times$2, [2,2]  \\
\hline
\end{tabular}
}

\pagebreak
\noindent\textbf{Box 26}

\resizebox{\textwidth}{!}{%
\begin{tabular}{|c|l|c|c|c|}
\hline
\textbf{Algebra} & \textbf{Schemes} & $\boldsymbol{\mathrm{dimdegPC}}(X)$ & $\boldsymbol{\mathrm{dimdegPC}}(X^{\mathrm{sing}})$ & $\boldsymbol{\mathrm{dimdegPC}}(\bar{X}^{\mathrm{sing}})$ \\
\hline
\multirow{4}{*}{$S(2,3)^{\sigma}$} 
  & \(X_{400}\)  & 0 & 0 &  0 \\
  & \(X_{020}\)  & 0 & 0 &  0 \\
  & \(X_{101}\)  & [4,10]$\times$2 & [3,6], [0,4] & [3,6] \\
  & \(X_{210}\)  & [4,56] & \textbf{\texttt{*****}}  &  [2,1], [1,1]$\times$7 \\
\hline
\end{tabular}
}

\noindent\textbf{Box 27}

\resizebox{\textwidth}{!}{%
\begin{tabular}{|c|l|c|c|c|}
\hline
\textbf{Algebra} & \textbf{Schemes} & $\boldsymbol{\mathrm{dimdegPC}}(X)$ & $\boldsymbol{\mathrm{dimdegPC}}(X^{\mathrm{sing}})$ & $\boldsymbol{\mathrm{dimdegPC}}(\bar{X}^{\mathrm{sing}})$ \\
\hline
\multirow{4}{*}{CentExt Type $B$}
  & \(X_{400}\)  & 0 & 0 &  0 \\
  & \(X_{020}\)  & 0 & 0 &  0 \\
  & \(X_{101}\)  & [4,10]$\times$2 & [3,6], [0,4] & [3,6] \\
   & \(X_{210}\)  & [4,56] & [1,5], [1,8], [1,2], [1,1]$\times$2  & [1,5], [1,8], [1,2], [1,1]$\times$2 \\
\hline
\end{tabular}
}


\noindent\textbf{Box 28}

\resizebox{\textwidth}{!}{%
\begin{tabular}{|c|l|c|c|c|}
\hline
\textbf{Algebra} & \textbf{Schemes} & $\boldsymbol{\mathrm{dimdegPC}}(X)$ & $\boldsymbol{\mathrm{dimdegPC}}(X^{\mathrm{sing}})$ & $\boldsymbol{\mathrm{dimdegPC}}(\bar{X}^{\mathrm{sing}})$ \\
\hline
\multirow{3}{*}{$L(1,1,2)^{\sigma}$} 
  & \(X_{400}\)  & 0 & 0 &  0 \\
  & \(X_{020}\)   & [3,2]$\times$2 & [2,2], [0,1]$\times$2 & [2,2], [0,1]$\times$2  \\
  & \(X_{101}\)   & [3,20] & [0,1]$\times$12 & [0,1]$\times$12  \\
\hline
\end{tabular}
}

\noindent\textbf{Box 29}

\resizebox{\textwidth}{!}{%
\begin{tabular}{|c|l|c|c|c|}
\hline
\textbf{Algebra} & \textbf{Schemes} & $\boldsymbol{\mathrm{dimdegPC}}(X)$ & $\boldsymbol{\mathrm{dimdegPC}}(X^{\mathrm{sing}})$ & $\boldsymbol{\mathrm{dimdegPC}}(\bar{X}^{\mathrm{sing}})$ \\
\hline
\multirow{3}{*}{ $A_5$} 
  & \(X_{400}\) & [2,1], [2,3] & [1,3], [0,8] & [1,3], [0,1] \\  
  & \(X_{020}\)   & [3,4] & [1,2], [0,3]$\times$2 & [1,1]\\
  & \(X_{101}\)  & [3,6]$\times$2, [3,4]$\times$2 & [2,3]$\times$2, [2,2]$\times$2, [2,1]$\times$2, [1,4]$\times$2  & [2,3]$\times$2, [2,2]$\times$2, [2,1]$\times$2 \\
\hline
\end{tabular}
}

\noindent\textbf{Box 30}

\resizebox{\textwidth}{!}{%
\begin{tabular}{|c|l|c|c|c|}
\hline
\textbf{Algebras} & \textbf{Schemes} & $\boldsymbol{\mathrm{dimdegPC}}(X)$ & $\boldsymbol{\mathrm{dimdegPC}}(X^{\mathrm{sing}})$ & $\boldsymbol{\mathrm{dimdegPC}}(\bar{X}^{\mathrm{sing}})$ \\
\hline
\multirow{3}{*}{%
\begin{tabular}{c}
\rule{0pt}{3.2ex}%
$\mathcal{F}_{\scriptscriptstyle\substack{(0,-1,-1,2)\\(0,0,0,0)}}$, \\
OreExt Type $S_1^{'}$
\end{tabular}}
  & \(X_{400}\)  & [2,1]$\times$2, [2,2] & [1,2], [1,1]$\times$3, [0,5] & [1,2], [1,1]$\times$3 \\
  & \(X_{020}\)   & [2,1]$\times$8 & [1,1]$\times$12 & [1,1]$\times$12 \\
  & \(X_{101}\)   & [2,1]$\times$8, [2,2]$\times$6  & [1,1]$\times$24 & [1,1]$\times$24 \\
\hline
\end{tabular}
}

\begin{remark}
The algebras $\mathcal{F}_{\scriptscriptstyle\substack{(0,-1,-1,2)\\(0,0,0,0)}}$ and OreExt Type $S_1^{'}$ are isomorphic under the linear change of generators 
$ \phi: x_1 \mapsto t, \;\; x_2 \mapsto \lambda x_2, \;\; x_3 \mapsto \mu x_1, \;\; x_4 \mapsto \nu x_3 $
where
$c^{2}d = \frac{q_{13}}{q_{14}^{2}}, \quad b^{2}d = \frac{1}{q_{13}}, \quad b c d = \frac{1}{q_{14}}, \quad \nu^2 = -\mu \lambda, \quad a = -q_{23}, \quad  \frac{\alpha}{a} = -q_{24}.$
\end{remark}

\noindent\textbf{Box 31}

\resizebox{\textwidth}{!}{%
\begin{tabular}{|c|l|c|c|c|}
\hline
\textbf{Algebra} & \textbf{Schemes} & $\boldsymbol{\mathrm{dimdegPC}}(X)$ & $\boldsymbol{\mathrm{dimdegPC}}(X^{\mathrm{sing}})$ & $\boldsymbol{\mathrm{dimdegPC}}(\bar{X}^{\mathrm{sing}})$ \\
\hline
\multirow{3}{*}{$\mathcal{F}_{\scriptscriptstyle\substack{(-1,-1,1,1) \\ (0,0,0,0)}}$} 
  & \(X_{400}\)  & [2,1]$\times$2, [2,2] & [1,1]$\times$5, [0,4]$\times$2 &[1,1]$\times$5 \\
  & \(X_{020}\)   & [2,4], [2,1]$\times$4 & [1,1]$\times$8, [0,2]$\times$4 & [1,1]$\times$8\\
  & \(X_{101}\)  & [2,1]$\times$8, [2,2]$\times$6  & [1,1]$\times$24 & [1,1]$\times$24 \\
\hline
\end{tabular}
}

\pagebreak
\noindent\textbf{Box 32}

\resizebox{\textwidth}{!}{%
\begin{tabular}{|c|l|c|c|c|}
\hline
\textbf{Algebra} & \textbf{Schemes} & $\boldsymbol{\mathrm{dimdegPC}}(X)$ & $\boldsymbol{\mathrm{dimdegPC}}(X^{\mathrm{sing}})$ & $\boldsymbol{\mathrm{dimdegPC}}(\bar{X}^{\mathrm{sing}})$ \\
\hline
\multirow{3}{*}{ $\mathcal{F}_{\scriptscriptstyle\substack{(0,-1,-1,2) \\ (-1,0,-1,2) \\ (0,0,0,0)}}$} 
  & \(X_{400}\)  & [2,1], [2,3] & [1,1]$\times$3, [0,7], [0,5]$\times$2 &[1,1]$\times$3 \\
  & \(X_{020}\)  & [2,1]$\times$8 & [1,1]$\times$12 & [1,1]$\times$12 \\
  & \(X_{101}\)   & [2,1]$\times$8, [2,2]$\times$6  & [1,1]$\times$24 & [1,1]$\times$24 \\
\hline
\end{tabular}
}

\noindent\textbf{Box 33}

\resizebox{\textwidth}{!}{%
\begin{tabular}{|c|l|c|c|c|}
\hline
\textbf{Algebra} & \textbf{Schemes} & $\boldsymbol{\mathrm{dimdegPC}}(X)$ & $\boldsymbol{\mathrm{dimdegPC}}(X^{\mathrm{sing}})$ & $\boldsymbol{\mathrm{dimdegPC}}(\bar{X}^{\mathrm{sing}})$ \\
\hline
\multirow{3}{*}{ $\mathcal{F}_{\scriptscriptstyle\substack{(0,-1,-1,2) \\ (0,0,0,0) \\ (-1,-1,2,0)}}$} 
  & \(X_{400}\)  & [2,1], [2,3] & [1,2], [1,1], [0,10], [0,6] &[1,2], [1,1] \\
  & \(X_{020}\)  & [2,1]$\times$8 & [1,1]$\times$12 & [1,1]$\times$12 \\
  & \(X_{101}\)   & [2,1]$\times$8, [2,2]$\times$6  & [1,1]$\times$24 & [1,1]$\times$24 \\
\hline
\end{tabular}
}

\noindent\textbf{Box 34}

\resizebox{\textwidth}{!}{%
\begin{tabular}{|c|l|c|c|c|}
\hline
\textbf{Algebras} & \textbf{Schemes} & $\boldsymbol{\mathrm{dimdegPC}}(X)$ & $\boldsymbol{\mathrm{dimdegPC}}(X^{\mathrm{sing}})$ & $\boldsymbol{\mathrm{dimdegPC}}(\bar{X}^{\mathrm{sing}})$ \\
\hline
\multirow{3}{*}{\begin{tabular}{c}
OreExt Type $A_1$, \\[1pt]
OreExt Type $A_2$, \\[1pt]
OreExt Type $A_3$ 
\end{tabular}} 
  & \(X_{400}\)  & [2,1], [2,3] & [1,3], [0,8] & [1,3], [0,1] \\
  & \(X_{020}\)   & [3,4] & [0,1]$\times$6 & [0,1]$\times$6 \\
  & \(X_{101}\) & [2,1]$\times$8, [2,2]$\times$6  & [1,1]$\times$24 & [1,1]$\times$24 \\
\hline
\end{tabular}
}

\begin{question}
Are the OreExt algebras \(A_1, A_2,\) and \(A_3\) isomorphic?
\end{question}

\noindent\textbf{Box 35}

\resizebox{\textwidth}{!}{%
\begin{tabular}{|c|l|c|c|c|}
\hline
\textbf{Algebra} & \textbf{Schemes} & $\boldsymbol{\mathrm{dimdegPC}}(X)$ & $\boldsymbol{\mathrm{dimdegPC}}(X^{\mathrm{sing}})$ & $\boldsymbol{\mathrm{dimdegPC}}(\bar{X}^{\mathrm{sing}})$ \\
\hline
\multirow{3}{*}{ OreExt Type $B_1$ } 
  & \(X_{400}\)  & [2,1], [2,3] & [1,3], [0,8] & [1,3], [0,1] \\
  & \(X_{020}\)   & [3,2]$\times$2 & [2,1]$\times$2, [1,4] & [2,1]$\times$2 \\
  & \(X_{101}\)   & [4,10]$\times$2 &  [3,6], [0,4] & [3,6] \\
\hline
\end{tabular}
}

\noindent\textbf{Box 36}

\resizebox{\textwidth}{!}{%
\begin{tabular}{|c|l|c|c|c|}
\hline
\textbf{Algebras} & \textbf{Schemes} & $\boldsymbol{\mathrm{dimdegPC}}(X)$ & $\boldsymbol{\mathrm{dimdegPC}}(X^{\mathrm{sing}})$ & $\boldsymbol{\mathrm{dimdegPC}}(\bar{X}^{\mathrm{sing}})$ \\
\hline
\multirow{3}{*}{\begin{tabular}{c}
OreExt Type $E_1$, \\[2pt]
OreExt Type $E_2$ 
\end{tabular}} 
  & \(X_{400}\)  & [2,1], [2,3] & [1,3], [0,8] & [1,3], [0,1] \\
  & \(X_{020}\)  & [2,6], [2,1]$\times$2 & [1,3]$\times$2 & [1,3]$\times$2 \\
  & \(X_{101}\) & [3,20] & [0,1]$\times$6 & [0,1]$\times$6\\
\hline
\end{tabular}
}

\begin{question}
Are the OreExt algebras \(E_1\) and \(E_2\) isomorphic?
\end{question}

\noindent\textbf{Box 37}

\resizebox{\textwidth}{!}{%
\begin{tabular}{|c|l|c|c|c|}
\hline
\textbf{Algebras} & \textbf{Schemes} & $\boldsymbol{\mathrm{dimdegPC}}(X)$ & $\boldsymbol{\mathrm{dimdegPC}}(X^{\mathrm{sing}})$ & $\boldsymbol{\mathrm{dimdegPC}}(\bar{X}^{\mathrm{sing}})$ \\
\hline
\multirow{3}{*}{\begin{tabular}{c}
OreExt Type $H.$I, \\[2pt]
OreExt Type $H.$II
\end{tabular}} 
  & \(X_{400}\) & [2,1], [2,3] & [1,3], [0,8] & [1,3], [0,1] \\
  & \(X_{020}\)  &  [2,2]$\times$3, [2,1]$\times$2 & [2,1]$\times$2, [1,5], [1,2]$\times$2, [1,1]$\times$2 & [2,1]$\times$2,[1,1]$\times$2 \\
  & \(X_{101}\)  & [3,6]$\times$2, [3,4]$\times$2 & [2,3]$\times$2, [2,2]$\times$2, [2,1]$\times$2, [1,4]$\times$2  & [2,3]$\times$2, [2,2]$\times$2, [2,1]$\times$2  \\
\hline
\end{tabular}
}

\begin{question}
Are the OreExt algebras $H.$I and $H.$II isomorphic?
\end{question}

\pagebreak
\noindent\textbf{Box 38}

\resizebox{\textwidth}{!}{%
\begin{tabular}{|c|l|c|c|c|}
\hline
\textbf{Algebra} & \textbf{Schemes} & $\boldsymbol{\mathrm{dimdegPC}}(X)$ & $\boldsymbol{\mathrm{dimdegPC}}(X^{\mathrm{sing}})$ & $\boldsymbol{\mathrm{dimdegPC}}(\bar{X}^{\mathrm{sing}})$ \\
\hline
\multirow{3}{*}{ OreExt Type $S_2$ } 
  & \(X_{400}\)  & [2,1]$\times$2, [2,2] & [1,2]$\times$2, [1,1]$\times$2 & [1,2]$\times$2, [1,1]$\times$2 \\
  & \(X_{020}\)   & [2,2]$\times$2, [2,1]$\times$4 & [1,1]$\times$10 & [1,1]$\times$10    \\
  & \(X_{101}\)   & [3,6]$\times$2, [3,4]$\times$2 &[2,3]$\times$4  &  [2,3]$\times$4  \\
\hline
\end{tabular}
}

\noindent\textbf{Box 39}

\resizebox{\textwidth}{!}{%
\begin{tabular}{|c|l|c|c|c|}
\hline
\textbf{Algebra} & \textbf{Schemes} & $\boldsymbol{\mathrm{dimdegPC}}(X)$ & $\boldsymbol{\mathrm{dimdegPC}}(X^{\mathrm{sing}})$ & $\boldsymbol{\mathrm{dimdegPC}}(\bar{X}^{\mathrm{sing}})$ \\
\hline
\multirow{3}{*}{ CentExt Type $H$ } 
  & \(X_{400}\)  & 0 & 0 &  0 \\
  & \(X_{020}\)   & [3,2]$\times$2 & [2,1]$\times$2, [1,4] & [2,1]$\times$2    \\
  & \(X_{101}\) & [3,6]$\times$2, [3,4]$\times$2 & [2,3]$\times$2, [2,2]$\times$2, [2,1]$\times$2, [1,4]$\times$2  & [2,3]$\times$2, [2,2]$\times$2, [2,1]$\times$2  \\
\hline
\end{tabular}
}

\clearpage
\pagebreak

\bibliographystyle{plain}
\bibliography{bibliography}

\end{document}